\documentclass[11pt,letterpaper,reqno]{amsart}
\RequirePackage[utf8]{inputenc}
\RequirePackage[T1]{fontenc}
\RequirePackage{xcolor}
\RequirePackage{graphicx}
\RequirePackage{hyperref}
\RequirePackage{amsmath, amssymb, amsfonts}
\usepackage[left=1in,right=1in,top=1in,bottom=1in,foot=0.5in]{geometry}
\RequirePackage{xspace}
\RequirePackage{enumerate}
\RequirePackage{mathrsfs}
\RequirePackage{hyphenat}

\PassOptionsToPackage{
    linesnumbered,
    ruled,
    noend,
    nofillcomment,
    english,
    onelanguage
}{algorithm2e}
\RequirePackage{algorithm2e}
\RequirePackage{xstring}

\RequirePackage{diagbox}

\RequirePackage{etoolbox}

\newenvironment{myFStyle}{}{}
\newenvironment{myFunction}{
    \smallskip
    \begin{myFStyle}
        \RestyleAlgo{plain}
        \LinesNotNumbered
        \SetAlgoNoLine
        \begin{algorithm}[H]
        }{
        \end{algorithm}
    \end{myFStyle}
}

\newtheorem{lemma}{Lemma}
\newtheorem{proposition}{Proposition}
\newtheorem{theorem}{Theorem}
\newtheorem{corollary}{Corollary}
\newtheorem{claim}{Claim}
\newtheorem{question}{Question}

\theoremstyle{definition}
\newtheorem{remark}{Remark}

\newcommand{\mcg}{\mathscr G} %                               famille de graphes
\newcommand{\mcf}{\mathscr F} %                                          famille
\newcommand{\NN}{\mathbb N} %                      ensemble des entiers naturels

\DeclareMathOperator{\St}{St} %               bouquet de cube autour d'un sommet
\renewcommand{\L}{\text{L}} %                                              label
\newcommand{\conc}{\circ} %                            concaténation de vecteurs
\newcommand{\id}{\text{id}} %                                        identifiant
\newcommand{\Ls}[2]{\text{L}_{#2}(#1)} %             label de #1 sur l'étoile #2
\newcommand{\LDt}[2]{\text{L}_{#2}(#1)} %  label (distance) de #1 sur l'arbre #2

\newcommand{\dir}{\text{dir}}
\newcommand{\dirU}{\text{dirU}}
\newcommand{\dirV}{\text{dirV}}
\newcommand{\lleft}{\text{1st}\xspace} %                              partie 1er
\newcommand{\rright}{\text{2nd}\xspace} %                             partie 2nd
\newcommand{\st}{\text{St}\xspace} %       partie relative à un bouquet de cubes
\newcommand{\med}{\text{Med}\xspace} %            identifiant d'un sommet médian
\newcommand{\distance}{\text{Dist}\xspace} %  distance une porte, un médian, ...
\newcommand{\rootS}{\text{Root}\xspace} %        racine d'une fibre sur l'étoile
\newcommand{\repD}{\text{Rep}\xspace} %        id dist d'arbre d'un représentant
\newcommand{\conv}{\rm conv\xspace}	

\DeclareMathOperator{\dist}{d} %                                        distance
\newcommand{\mproj}[2]{\text{Pr}{(#1,#2)}} %                 projection métrique
\newcommand{\TODO}[1][]{%
    \begin{center}%
        \textcolor{red}{\textbf{...TODO}%
        \ifthenelse{\equal{#1}{\empty}}{}{\textit{(#1)}}%
        \textbf{...}}%
    \end{center}%
}

\newcommand{\dls}{\text{DLS}\xspace} %                  distance labeling scheme
\newcommand{\tc}{\textbf{(TC)}\xspace} %                      triangle condition
\newcommand{\qc}{\textbf{(QC)}\xspace} %                    quadrangle condition

\newcommand{\funcName}[1]{%
        \texttt{#1}%
    \xspace%
}
\newcommand{\procName}[3][]{%
    \ifthenelse{\equal{#1}{\empty}}{%
        \hyperref[#2]{%
            \textsc{#3}%
        }%
    }{% #1 = nolink
        \textsc{#3}%
    }\xspace%
}

\def\/{\discretionary{}{}{}} % permet une césure sans "-" pour les noms de
\newcommand{\+}{\_\allowbreak}
\newcommand{\distEncTree}{\funcName{Enc\+Tree}}
\newcommand{\distDecTree}{\funcName{Dist\+Tree}}

\SetKwFunction{DistPCNeighboring}{Dist\+1pc-neighboring}
\SetKwFunction{DistCCNeighboring}{Dist\+2cc-neighboring}

\newcommand{\encStarBri}{\funcName{Enc\+Star}}
\newcommand{\distEncBri}[1][]{\procName[#1]{alg_distencBri}{Enc\+Dist}}
\newcommand{\distDecBri}[1][]{\procName[#1]{alg_distdecBri}{Distance}}

\SetKwInput{KwData}{Input}
\SetKwInput{KwResult}{Output}
\let\Input\KwData
\let\Output\KwResult
\SetKw{Stop}{stop}
\SetKwProg{Fn}{function}{:}{}

\definecolor{SeaGreen}{RGB}{46,139,87}
\definecolor{mygreen}{RGB}{0,180,0}

\SetCommentSty{commcolor}
\SetArgSty{textnormal}

\title{%
    Distance labeling schemes for $K_4$-free bridged graphs%
    }
\author{Victor Chepoi, Arnaud Labourel, and S\'ebastien Ratel}
\date{}
\thanks{The short version of this paper appeared at SIROCCO 2020}

\begin{document}

\begin{abstract}
    $k$-Approximate distance labeling schemes are schemes that label the
    vertices of a graph with short labels in such a way that the
    $k$-approximation of the distance between any two vertices $u$ and $v$ can
    be determined efficiently by merely inspecting the labels of $u$ and $v$,
    without using any other information. One of the important problems is
    finding natural classes of graphs admitting exact or approximate distance
    labeling schemes with labels of polylogarithmic size.
    In this paper, we describe a $4$-approximate distance labeling scheme for
    the class of $K_4$-free bridged graphs. This scheme uses labels of
    poly-logarithmic length $O(\log n^3)$ allowing a constant decoding time.
    Given the labels of two vertices $u$ and $v$, the decoding function returns
    a value between the exact distance $d_G(u,v)$ and its quadruple $4d_G(u,v)$.
\end{abstract}

\maketitle

\begin{small}
    \centerline{Aix Marseille Univ, Université de Toulon, CNRS, LIS, Marseille,
    France}

    \centerline{\texttt{\{victor.chepoi, arnaud.labourel,
    sebastien.ratel\}@lis-lab.fr}}
\end{small}

\section{Introduction}

\subsection{Distance labeling schemes} 

A \emph{(distributed) labeling scheme} is a way of distributing the global 
representation of a graph over its vertices, by giving them some local 
information. The goal is  to be able to answer specific queries using only 
local knowledge. Peleg \cite{Peleg00} gave a wide survey of the importance of 
such localized data structures in distributed computing and of the types of 
queries based on them. Such queries can be of various types (see \cite{Peleg00}
for a comprehensive list), but adjacency, distance, or routing can be listed 
among the most fundamental ones. The quality of a labeling scheme is measured 
by the size of the labels of vertices and the time required to answer the 
queries. Adjacency, distance, and routing labeling schemes for general graphs 
need labels of linear size. Labels of size $O(\log n)$ or $O(\log^2 n)$ are
sufficient for such schemes on trees.
Finding natural classes of graphs admitting distributed labeling schemes with 
labels of polylogarithmic size  is an important and challenging problem.

In this paper we investigate distance labeling schemes. A \emph{distance 
labeling scheme} (\emph{\dls} for short) on a graph family $\mcg$ consists of 
an \emph{encoding} function $C_G : V \to \{0,1\}^*$ that gives binary labels to 
vertices of a graph $G \in \mcg$ and of a \emph{decoding} function $D_G : 
\{0,1\}^* \times \{0,1\}^* \to \NN$ that, given the labels of two vertices $u$ 
and $v$ of $G$, computes the distance $\dist_G(u,v)$ between $u$ and $v$ in 
$G$. For $k \in \NN^*$, we call a labeling scheme a \emph{$k$-approximate 
distance labeling scheme} if the decoding function computes an integer 
comprised between $\dist_G(u,v)$ and $k \cdot \dist_G(u,v)$.
Finding natural classes of graphs admitting exact or approximate distance 
labeling schemes with labels of polylogarithmic size is an important and 
challenging problem. In this paper we continue the line of research we started 
in \cite{ChLaRa_labeling_median} to investigate classes of graphs with rich 
metric properties, and we design approximate distance labeling schemes of 
polylogarithmic size for $K_4$-free bridged graphs.

\subsection{Related work}

By a result of Gavoille et al. \cite{GaPePeRa_distance}, the family of all 
graphs on $n$ vertices admits a distance labeling scheme with labels of $O(n)$ 
bits.
This scheme is asymptotically optimal because simple counting arguments on the 
number of $n$-vertex graphs show that $\Omega(n)$ bits are necessary. Another 
important result is that trees admit a \dls with labels of $O(\log^2 n)$ bits. 
Recently Freedman et al. \cite{FrGaNiWe_trees} obtained such a scheme allowing 
constant time distance queries. Several graph classes containing trees also 
admit \dls with labels of length $O(\log^2 n)$: bounded tree-width graphs
\cite{GaPePeRa_distance}, distance-hereditary graph \cite{GaPa_decomposition},
bounded clique-width graphs \cite{CoVa_bounded_cw}, or planar graphs of 
non-positive curvature \cite{ChDrVa_labeling}. More recently, in 
\cite{ChLaRa_labeling_median} we designed \dls with labels of length $O(\log^3 
n)$ for cube-free median graphs.
Other families of graphs have been considered such as interval graphs, 
permutation graphs, and their generalizations 
\cite{BaGa_data_structures,GaPa_interval_graphs} for which an optimal bound of
$\Theta(\log n)$ bits was given, and planar graphs for which there is a lower
bound of $\Omega(n^{\frac{1}{3}})$ bits \cite{GaPePeRa_distance} and an upper
bound of $O(\sqrt{n})$ bits \cite{GaUz_distance}. Other results concern 
approximate distance labeling schemes. 
For arbitrary graphs, Thorup and Zwick \cite{Thorup_approx_distance} proposed 
$(2k-1)$-approximate \dls, for each integer $k \geq 1$, with labels of size
$O(n^{1/k} \log^2 n)$.
In \cite{GaKaKaPaPe_approx}, it is proved that trees (and bounded tree-width 
graphs as well) admit $(1+1/\log n)$-approximate \dls with labels of size 
$O(\log n \log\log n)$, and this is tight in terms of label length and 
approximation.
They also designed $O(1)$-additive \dls with $O(\log^2 n)$-labels for several
families of graphs, including the graphs with bounded longest induced cycle,
and, more generally, the graphs of bounded tree–length. Interestingly, it is
easy to show that every exact \dls for these families of graphs needs labels
of $\Omega(n)$ bits in the worst-case \cite{GaKaKaPaPe_approx}.

\subsection{Bridged graphs}

Together with hyperbolic, median, and Helly graphs, bridged graphs constitute 
the most important classes of graphs in metric graph theory 
\cite{BaCh_survey,ChChHiOs}. They occurred in the investigation of graphs 
satisfying basic properties of classical Euclidean convexity:
{\it bridged graphs} are the graphs in which the neighborhoods of convex sets 
are convex and it was shown in  \cite{FaJa_convexity,SoCh_convexification} that
they are exactly the graphs in which all isometric cycles have length 3.
Bridged graphs represent a far-reaching generalization of chordal graphs.
From the point of view of structural graph theory, bridged graphs are quite 
general and universal: any graph $H$ not containing induced $C_4$ and $C_5$ may
occur as an induced subgraph  of a bridged graph. However, from the metric point
of view, bridged graphs have a rich structure. This structure was thoroughly 
studied and used in geometric group theory and in metric graph theory. 

The convexity of balls around convex sets and the uniqueness of geodesics 
between pairs of points are two basic properties not only of Euclidian or 
hyperbolic geometries but also of all CAT(0) geometries.  Introduced by Gromov 
in his seminal paper \cite{Gr}, CAT(0) (alias nonpositively curved) geodesic 
metric spaces are fundamental objects of study in metric geometry and geometric 
group theory.
Graphs with strong metric properties often arise as 1-skeletons of CAT(0) cell 
complexes. Gromov \cite{Gr} gave a nice combinatorial local-to-global 
characterization of CAT(0) cube complexes. 
Based on this result, Chepoi \cite{Ch_CAT} established a bijection between the
1-skeletons of CAT(0) cube complexes and the median graphs, well-known in 
metric graph theory \cite{BaCh_survey}. A similar characterization of CAT(0) 
simplicial complexes with regular Euclidean simplices as cells seems 
impossible. Nevertheless, Chepoi \cite{Ch_CAT} characterized the simplicial 
complexes having bridged graphs as 1-skeletons as the simply connected 
simplicial complexes in which the  neighborhoods of vertices do not contain 
induced 4- and 5-cycles.  Januszkiewicz and Swiatkowski \cite{JaSw}, Haglund 
\cite{Haglund},  and Wise \cite{Wise} rediscovered this class of simplicial 
complexes, called them {\it systolic complexes},  and used them in the context 
of geometric group theory.  Systolic complexes are contractible 
\cite{Ch_CAT,JaSw} and they are considered as  good combinatorial analogs of
CAT(0) metric spaces. One of the main results of \cite{JaSw} is that systolic 
groups (i.e., groups acting geometrically on systolic complexes) have the 
strong property of biautomaticity, which means that their Cayley graphs admit 
families of paths which define a regular language. The papers 
\cite{ChepoiOsajda, Elsner2009-isometries, Haglund, HoOs, HuOs_ms, HuOs_Artin, 
JaSw, JanuszkiewiczSwiatkowski2007, OsaPrzy2009, Prytula2017, Pr3,Wise}
represent a small sample of papers on systolic complexes and of groups acting 
on them.  Bridged graphs have also been investigated in several 
graph-theoretical  papers; cf. \cite{AnFa, BaCh_Helly, Ch_bridged, Po_bridged1, 
Po_bridged2} and the survey \cite{BaCh_survey}. In particular, in 
\cite{AnFa,Ch_bridged} was shown that bridged graphs are dismantlable (a 
property stronger that contractibility of clique complexes), which implies that 
bridged graphs are cop-win.
At the difference of median graphs (which occur as domains of event structures 
in concurrency, as solution sets of 2-SAT formulas in complexity, and as 
configuration spaces in robotics),  bridged graphs have not been extensively 
studied or used in full generality in Theoretical Computer Science. Notice 
however the papers \cite{ChDrVa_SODA,ChFaVa}, where linear time algorithms for 
diameter, center, and median problems were designed for planar bridged  graphs 
(called \emph{trigraphs}), i.e., planar graphs in which all inner faces are 
triangles and all inner vertices have degrees $\ge 6$. The trigraphs were 
introduced in \cite{BaCh_dwmg} as building stones in the decomposition theorem 
of weakly median graphs. The papers \cite{ChDrVa_wn,ChDrVa2005} designed for 
trigraphs exact distance and routing labeling schemes with labels of 
$O(\log^2n)$ bits.

A {\it $K_4$-free bridged graph} is a bridged graph not containing 4-cliques.  
The triangular grid is the simplest example of a trigraph and any trigraph is a 
$K_4$-free bridged graph. In Fig. \ref{fig_K4-free_bridged_graphs} we present
two examples of $K_4$-free bridged graphs, which are not trigraphs. 
Since the clique-number of a bridged graph $G$ is equal to the topological dimension of its clique complex plus one,
$K_4$-free bridged graphs are exactly the 1-skeletons of two-dimensional systolic complexes. Such complexes have been investigated
in geometric group theory in the papers \cite{GeSh,HoOs}. From the point of view of structural graph theory, $K_4$-free bridged graphs are quite
general: any graph of girth $\ge 6$ may occur in the neighborhood of a vertex of a
$K_4$-free bridged graph (and any graph not containing induced $C_4$ and $C_5$ may occur in the neighborhood of a vertex of a bridged graph). Weetman \cite{Weetman} described a nice construction of (infinite) graphs in which
the neighborhoods of all vertices are isomorphic to a prescribed finite graph of girth $\ge 6$.  From the local-to-global characterization of bridged graphs of
 \cite{Ch_CAT} it follows that the resulting graphs are $K_4$-free bridged graphs. Note also that $K_4$-free bridged graphs may contain any complete graph $K_n$ as a minor.
 To see this, it suffices to glue together $\frac{n(n-1)}{2}$ copies of equilateral triangles with enough large but identical side (say, side $2n$) of the triangular grid as we did
 in the left part of Fig. \ref{fig_K4-free_bridged_graphs} that contains $K_6$ as a minor (subdivision of $K_6$ indicated by edges in blue).

\subsection{Our results} We continue with the main result of our paper: 

\begin{theorem} \label{thm_dist_labeling_bridged}
The class $\mcg$ of $K_4$-free bridged graphs on $n$ vertices admits a
$4$-approximate distance labeling scheme using labels of $O(\log^3 n)$
bits. These labels are constructed in polynomial $O(n^2 \log n)$ time and
can be decoded in constant time, assuming that the distance matrix of $G$ is
provided.
\end{theorem}

The remaining part of this paper is organized in the following way. The main ideas of our distance labeling scheme
are informally described in Section \ref{sect_main_ideas_dim2_bridged}. Section
\ref{sect_prelim} introduces the notions used in this paper.
The next three Sections \ref{sect_properties_dim2_bridged}, \ref{sect_stars_in_dim2_bridged}, and \ref{sect_boundaries}
present the most important
geometric and structural properties of $K_4$-free graphs, which are the
essence of our distance labeling scheme.  In particular, we describe a partition of vertices of $G$ defined by the star of a median vertex.
In Section \ref{sect_classif_vertices_dim2_bridged} we characterize the pairs of vertices connected by a shortest
path containing the center of this star. The distance labeling scheme and its 
performances are described in Section \ref{sect_dist_labeling_dim2_bridged}.

\begin{figure}[htb]
    \centering
    \begin{minipage}{0.49\linewidth}
        \centering
        \includegraphics[width=0.5\linewidth]{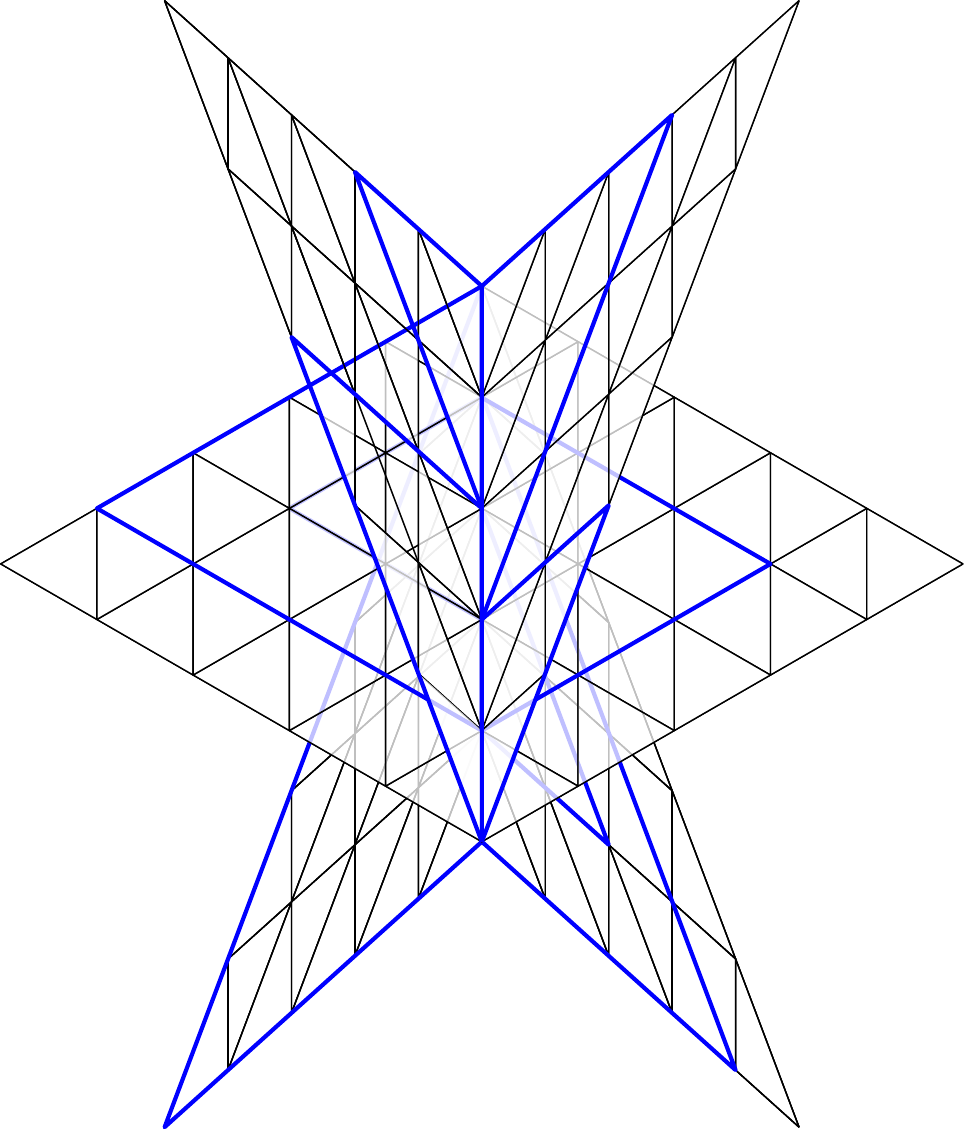}
    \end{minipage}
    \begin{minipage}{0.49\linewidth}
        \centering
        \includegraphics[width=0.8\linewidth]{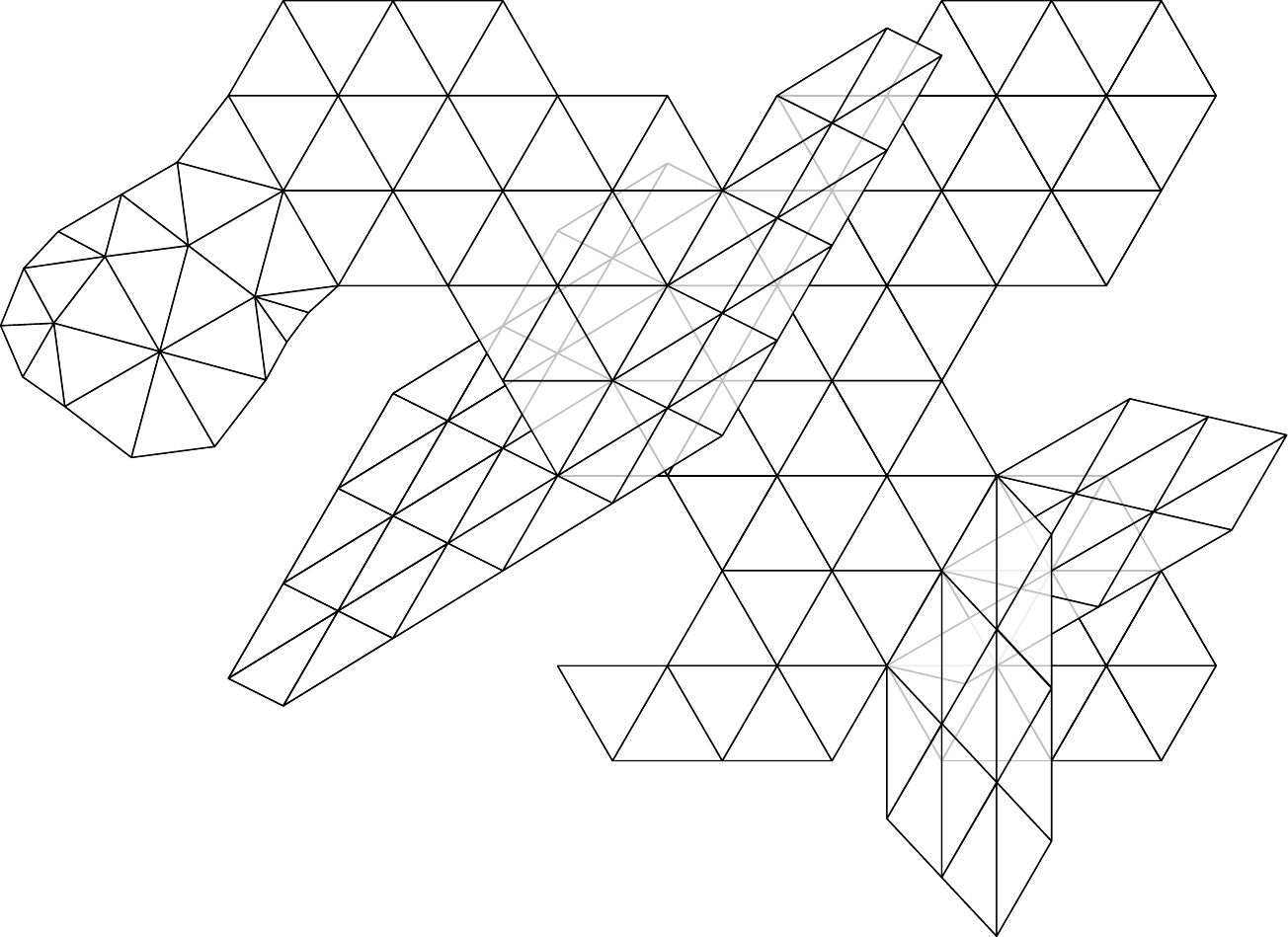}
    \end{minipage}
    \caption{
        \label{fig_K4-free_bridged_graphs}
        Examples of $K_4$-free bridged graphs.
    }
\end{figure}

\section{Main ideas of the scheme}
\label{sect_main_ideas_dim2_bridged}

The global structure of our distance labeling scheme for $K_4$-free
bridged graphs is similar to the one
described in \cite{ChLaRa_labeling_median} for cube-free median graphs.
Namely, the scheme is based on a recursive partitioning of the graph into a
star and its fibers (which are classified as panels and cones).
However, the stars and the fibers of $K_4$-free bridged graphs have completely
different structural and metric properties from those of cube-free median
graphs. Therefore, the technical tools  are completely different from those
used in \cite{ChLaRa_labeling_median}.

Let $G = (V,E)$ be a $K_4$-free bridged graph with $n$ vertices. The encoding
algorithm first searches for a \emph{median vertex} $m$ of $G$, i.e., a vertex
minimizing $x \mapsto \sum_{v \in V} \dist_G(x,v)$. It then computes a particular
$2$-neighborhood of $m$ that we call the \emph{star}
$\St(m)$ of $m$ such that  
every vertex of $G$ is assigned to a unique vertex of $\St(m)$. This
allows us to define the fibers of $\St(m)$: for a vertex $x \in \St(m)$, the
\emph{fiber} $F(x)$ of $x$ corresponds to the set of all the vertices of $G$
associated to $x$. The set of all  fibers of $\St(m)$ defines a partitioning of
$G$. Moreover, choosing $m$ as a median vertex ensures that every fiber contains at
most half of the vertices of $G$. Up to this point the scheme is the same as in
\cite{ChLaRa_labeling_median}, but here come the first differences. Namely, the
fibers are not convex, nevertheless, they
are connected and isometric, and thus induce bridged subgraphs of $G$.
Consequently, we can apply  recursively the partitioning to every
fiber, without accumulating errors on distances at each step. Finally, we study
the \emph{boundary}  and the \emph{total boundary} of each fiber, i.e., respectively,
the set of all vertices of a fiber having a
neighbor in another fiber and the union of all boundaries of a
fiber. We will see that those boundaries do not induce actual trees but something
close that we call ``starshaped trees''.
Unfortunately, these starshaped trees are not isometric. This explains why we obtain
an approximate and not an exact distance labeling scheme.
Indeed, distances computed in the total boundary can be twice as much as the
distances in the graph.

We distinguish two types of fibers $F(x)$ depending on the distance between $x$
and $m$: \emph{panels} are fibers leaving from a neighbor $x$ of $m$; \emph{cones} are fibers
associated to a vertex $x$ at distance $2$ from $m$.
One of our main results towards obtaining  a compact labeling scheme establishes that a
vertex in a panel admits two ``exit'' vertices on the total boundary of this
panel and that a vertex in a cone admits one ``entrance'' vertex on each boundary of
the cone (and it appears that every cone has exactly two boundaries).
The median vertex $m$ or those ``entrances'' and ``exits'' of a vertex $u$ on a
fiber $F(x)$ are guaranteed to lie on a path of length at most four times a
shortest $(u,v)$-path for any vertex $v$ outside $F(x)$. %It follows that,
At each recursive step, every vertex $u$ has to store
information relative only to three vertices ($m$, and the two ``entrances'' or
``exits'' of $u$).
Since a panel can have a linear number of boundaries, without this main
property, our scheme would use labels of linear length because a vertex in a
panel could have to store information relative to each boundary to allow to
compute distances with constant (multiplicative) error. Since we allow a multiplicative
error of $4$ at most, we will see that in
almost all cases, we can return the length of a shortest $(u,v)$-path passing
through the center $m$ of the star $\St(m)$ of the partitioning at some
recursive step. 
Lemmas \ref{lem_separated_vertices_dim2_bridged} and
\ref{lem_quasi-separated_vertices_dim2_bridged} indicate when this
length corresponds to the exact distance between $u$ and $v$ and when it is an
approximation of this distance. 
The case where $u$ and $v$ belong to distinct fibers that are ``too close'' are
more technical and are the one leading to a multiplicative error of 4.

\section{Preliminaries}
\label{sect_prelim}

\subsection{General notions} All graphs $G = (V,E)$ in this note are finite, undirected, simple, and
connected. We write $u\sim v$ if two vertices $u$ and $v$ are adjacent.
For a subset $A$ of vertices of $G$, we denote by $G[A]$ the subgraph of $G$
induced by $A$.
The \emph{distance} $d_G(u,v)$ between $u$ and $v$ is the length of
a shortest $(u,v)$-path in $G$. For a subset $A$ of $V$ and for two vertices
$u, v \in A$, we denote by $\dist_A(u,v)$ the distance
$\dist_{G[A]}(u,v)$. The \emph{interval} $I(u,v)$ between $u$
and $v$ consists of all the vertices on shortest $(u,v)$--paths. In other
words, $I(u,v)$ denotes all the vertices (metrically) \emph{between} $u$ and
$v$: $I(u,v):=\{ w \in V: \dist_G(u,w) + \dist_G(w,v) = \dist_G(u,v)\}$.
Let $H=(V',E')$ be a subgraph of $G$. Then $H$ is called  \emph{convex} if
$I(u,v) \subseteq H$ for any two vertices $u,v$ of $H$.
The {\it convex hull} of a subgraph $H$ of $G$ is the smallest convex subgraph $\conv(H)$ containing $H$.
A connected subgraph $H$ of $G$ is called \emph{isometric}
if $d_H(u,v)=d_G(u,v)$ for any vertices $u,v$ of $H$; if in addition $H$ is a cycle of $G$,
we call $H$ an {\it isometric cycle}.
For a vertex $x$ and a set of vertices $A \subseteq V$, let $\dist_G(x,A) :=
\min \{ \dist_G(x,a) : a \in A \}$ be the \emph{distance for $x$ to $A$}.
The \emph{metric projection} of a vertex $x \in V$ on a set $A \subseteq V$ (or
on $G[A]$)
is the set $\mproj{x}{A} := \{ y \in A : \dist_G(x,y) = \dist_G(x,A) \}$.
The \emph{neighborhood of  $A$} in $G$
is the set $N[A]:= A\cup \{ v \in V\setminus A : \exists u\in A, v\sim u\}$.
The \emph{ball} of radius $k$
centered at $A$ is the set $B_k(A):=\{ v: d_G(v,A)\le k\}$.
When $A$ is a singleton $a$,
then $N[a]$ is the closed neighborhood of $a$ and $B_k(a) := B_k(A)$. Note that $B_1(a)=N[a]$.
The \emph{sphere} of radius $k$
centered at $a$ is the set $S_k(A):=\{ v: d_G(v,A)= k\}$.

\subsection{Bridged graphs and metric triangles}
A graph $G$ is \emph{bridged} if any isometric cycle of $G$ has length 3. As
shown in  \cite{FaJa_convexity,SoCh_convexification},
bridged graphs are characterized by one of the fundamental properties of
CAT(0) spaces: \textit{the neighborhoods of all convex subgraphs of a
bridged graph are convex}. Consequently, balls in bridged graphs are convex.
Bridged graphs constitute an important subclass of weakly modular graphs: a
graph family that unifies numerous interesting classes of metric graph theory
through ``local-to-global'' characterizations \cite{ChChHiOs}.
\emph{Weakly modular graphs} are the graphs satisfying the
following \emph{quadrangle \qc} and \emph{triangle \tc conditions} \cite{BaCh_Helly,Chepoi_metric_triangles}:
\begin{quote}
    \qc $\forall u, v, w, z \in V$ with $k := \dist_G(u, v) = \dist_G(u, w)$,
    $\dist_G(u, z) = k + 1$, and $vz, wz \in E$, $\exists x \in V$ s.t.
    $\dist_G(u,x) = k - 1$ and $xv, xw \in E$.
    \\[1ex]
    \tc $\forall u,v,w \in V$ with $k := \dist_G(u,v) = \dist_G(u,w)$, and $vw
    \in E$, $\exists x \in V$ s.t. $\dist_G(u,x) = k - 1$ and $xv, xw \in E$.
\end{quote}
Bridged graphs are \emph{exactly the weakly modular graphs with no induced
cycle of length $4$ or $5$} \cite{Chepoi_metric_triangles}. Observe that, since
bridged graphs do not contain induced $4$-cycles, the quadrangle condition is
implied by the triangle condition.

A \emph{metric triangle} $u_1u_2u_3$ of a  graph $G = (V,E)$ is a
triplet $u_1,u_2,u_3$ of vertices such that for every $(i,j,\ell) \in
\{1,2,3\}^3$, $I(u_i,u_j) \cap I(u_j,
u_\ell) = \{u_j\}$ \cite{Chepoi_metric_triangles}. A metric triangle
$u_1u_2u_3$ is \emph{equilateral of size $k$} if
$\dist_G(u_1,u_2) = \dist_G(u_2,u_3) = \dist_G(u_1,u_3) = k$. If $k = 0$, then
the metric triangle consists of a single vertex $u_1=u_2=u_3$, and if $k = 1$
then the vertices $u_1$, $u_2$ and $u_3$ are pairwise adjacent.
A metric triangle $u_1u_2u_3$ is \emph{strongly equilateral} if any $x\in
I(u_1,u_2)$, the equality $\dist_G(u_3,x) = \dist_G(u_1,u_2)$ holds.
 Weakly modular graphs can be characterized via metric triangles in the
 following way :

\begin{proposition}\cite{Chepoi_metric_triangles}
    \label{metric_triangles}
    A graph $G$ is weakly modular if and only if any metric triangle of $G$ is
    strongly equilateral.
\end{proposition}

In particular, every metric triangle of a weakly modular graph is equilateral.
A metric triangle $u_1'u_2'u_3'$ is called a {\it quasi-median} of a triplet
$u_1,u_2,u_3$ if for each pair
$1\le i<j\le 3$ there exists a shortest $(u_i,u_j)$-path passing via $u_i'$ and
$u_j'$. Each triplet  $u_1,u_2,u_3$ of vertices of any graph $G$
admits at least one
quasi-median: it suffices to take as $u_1'$ a furthest from $u_1$ vertex from
$I(u_1,u_2)\cap I(u_1,u_3)$, as $u_2'$ a furthest from $u_2$ vertex from
$I(u_2,u_1') \cap I(u_2,u_3)$, and as $u_3'$ a furthest from $u_3$ vertex from
$I(u_3,u_1')\cap I(u_3,u_2')$.

\subsection{Gauss-Bonet formula}
We conclude this section with the classical  Gauss-Bonnet formula, which will 
be useful in some our proofs. Let $G = (V,E)$ be a plane graph and let 
$\partial G$ denote the cycle
delimiting its outer face. We view the inner faces of $G$  of length $k$ as regular $k$-gons of the Euclidean plane;
each of their angles must be equal to $\frac{k-2}{k} \pi$. For all vertices $v$ of  $G$, let $\alpha(v)$ denote the
sum of the angles of the inner faces of $G$ containing $v$.  In other words, if  $\mathscr C(v)$ denotes the set of all inner faces
of $G$ containing $v$, then
$$
\alpha(v) := \sum_{C \in \mathscr C(v)} \frac{|V(C)| - 2}{|V(C)|} \pi.
$$
For all $v  \in \partial G$, we set $\tau(v) := \pi - \alpha(v)$, and for all inner vertices
$v$ of $G$ we set $\kappa(v) := 2\pi - \alpha(v)$.
The parameters $\kappa(v)$ and $\tau(v)$ measure the ``defect'' of
the angles around $v$ (i.e., the gap between the actual value $\alpha(v)$ of
the angles around $v$, and the value $\pi$ or $2\pi$ that should be the correct
one if the polygons were ``really embeddable'' in the Euclidean plane).
A discrete version of Gauss-Bonnet's Theorem (see \cite{LySc_group_theory})
establishes the following formula (an example is given on Fig. \ref{fig_GB}):

\begin{theorem}[Gauss-Bonnet]
    \label{thm_Gauss-Bonnet}
    Let $G = (V,E)$ be a planar graph. Then,
    $$
    \sum_{v \in \partial G} \tau(v) + \sum_{v \in V \setminus \partial
        G} \kappa(v) = 2\pi.
    $$
\end{theorem}

\begin{figure}
    \centering
    \includegraphics[width=0.35\linewidth]{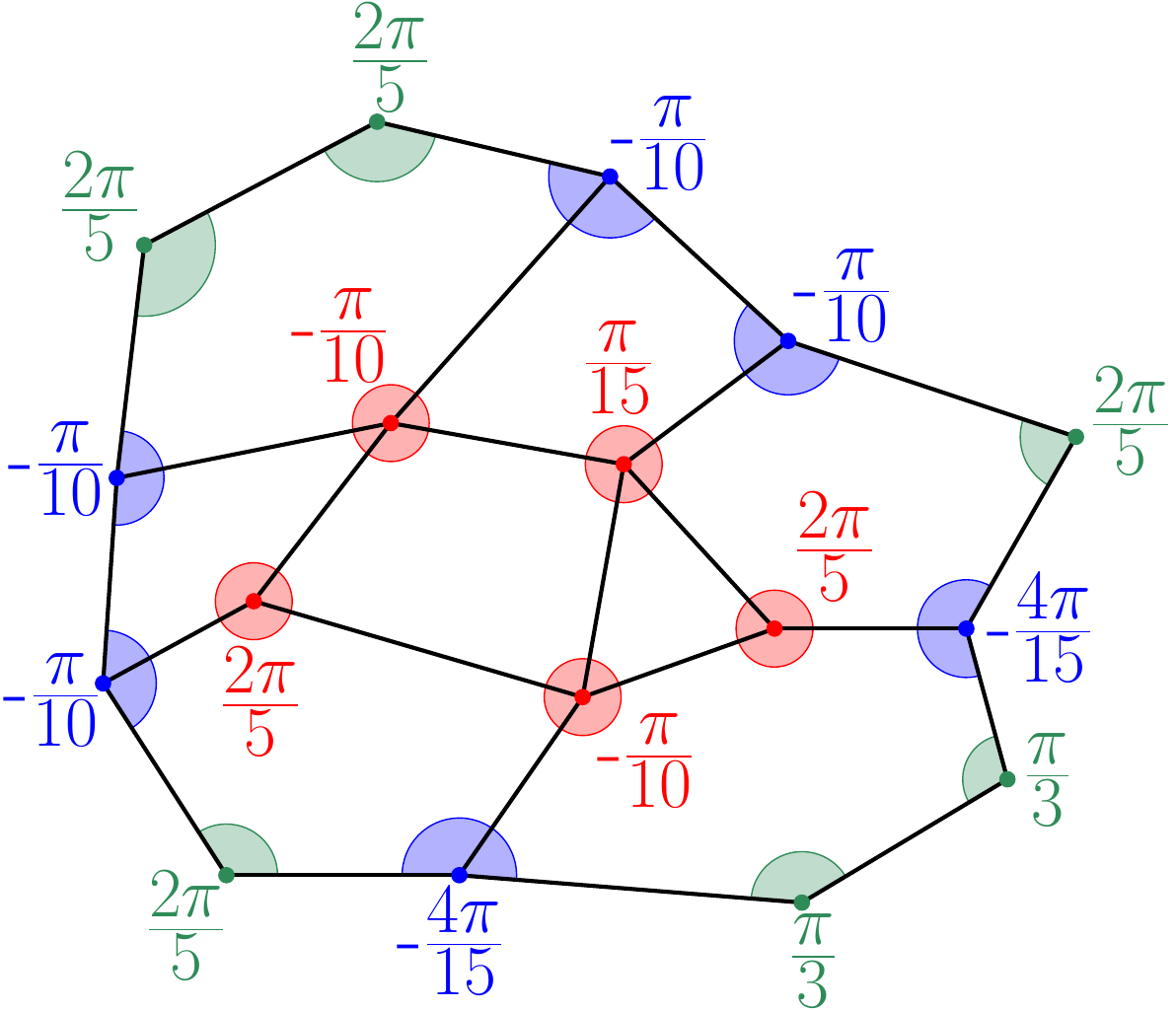}
    \caption{
        \label{fig_GB}
        Gauss-Bonnet's Theorem, and values of $\tau$ and $\kappa$ for the
        vertices of a planar graph $G$.
        The displayed values are those of $\kappa$ for the inner vertices (in
        red), and those of $\tau$ for the others (corners in green, and
        remaining vertices in blue).
        $ \textcolor{blue}{     4 \left( -\frac{\pi}{10}  \right)}
        + \textcolor{blue}{     2 \left( -\frac{4\pi}{15} \right)}
        + \textcolor{SeaGreen}{ 4 \left(  \frac{2\pi}{5}  \right)}
        + \textcolor{SeaGreen}{ 2 \left(  \frac{\pi}{3}   \right)}
        + \textcolor{red}{      2 \left( -\frac{\pi}{10}  \right)}
        + \textcolor{red}{      2 \left(  \frac{2\pi}{5}  \right)}
        + \textcolor{red}{                \frac{\pi}{15}}
        = 2\pi$.
    }
\end{figure}

Appendix~\ref{ref_glo} contains a glossary with all notions and notations.

\section{Metric triangles and intervals}
\label{sect_properties_dim2_bridged}

\subsection{Flat triangles and burned lozenges}
From now on we suppose that $G$ is a $K_4$-free bridged graph.
The \emph{triangular grid} is a tiling of the plane with equilateral triangles
of side $1$.
A \emph{flat triangle} is an equilateral triangle in the triangular grid;
for an illustration see Fig. \ref{fig_flat_triangle} (left).
The interval $I(u,v)$ between two vertices $u,v$ of the triangular grid at distance $\ell$ induces a lozenge (see Fig.
\ref{fig_grille_rongee}, right).  A \emph{burned lozenge} is obtained from $I(u,v)$ by
iteratively removing vertices of degree $3$; equivalently, a burned lozenge is
the subgraph of $I(u,v)$ in the region bounded by two shortest
$(u,v)$-paths. The vertices of a burned lozenge are naturally classified into
\emph{border} and \emph{inner} vertices.
Border vertices can be articulation points of the graph defined by $I(u,v)$,
see Figure \ref{fig_degenerated_vertices}.
A non-articulation border vertex is called a \emph{convex corner} if it belongs
to two triangles of $I(u,v)$, and it is called a \emph{concave corner} if it
belongs to four triangles, see Figure \ref{fig_grille_rongee}(right). The
remaining non-articulation border vertices belong to three triangles and are
just called \emph{regular borders}.
Those vertices and the convex corners are vertices of local convexity of the
burned lozenge, while concave corners are vertices of local concavity.
A \emph{halved burned lozenge} is the intersection of a burned lozenge with
the ball $B_k(u)$, where $0\le k\le \ell$. Notice that all spheres $S_i(u)$
induce parallel paths of the triangular grid.

We denote the convex hull of a metric triangle $uvw$ of $G$ by $\Delta(u,v,w)$ 
and call it a {\it deltoid}.

We start with the following auxiliary result:
\begin{lemma} \label{K3} Spheres $S_k(u)$ of $G$ cannot contain triangles $K_3$.
\end{lemma}

\begin{proof} Suppose by way of contradiction that the vertices $x_1,x_2,x_3\in S_k(u)$ induce a $K_3$. By triangle condition, there exists a vertex $y$ adjacent to $x_1,x_2$ at distance
$k-1$ from $u$. For the same reason, there exists a vertex $z$ adjacent to $x_2,x_3$ at distance $k-1$ from $u$. Since $G$ does not contain induced $K_4$, $y\ne z$ and $y\nsim x_3, z\nsim x_1$.
Since $y,z\in B_{k-1}(u)$ and $x_2\notin B_{k-1}(u)$, by the convexity of the ball $B_{k-1}(u)$, we have $y\sim z$. But then the vertices $y,z,x_3,x_1$ induce a forbidden 4-cycle.
\end{proof}

The following two lemmas were known before for $K_4$-free planar bridged graphs (see for example, [Proposition 3] \cite{BaCh_wma}
for the first lemma) but their proofs remain the same. For their full proofs, see \cite{Ratel}.

\begin{lemma}
    \label{lem_MT_structure_K4-free}
    Any deltoid $\Delta(u,v,w)$ of $G$ is a flat triangle.
\end{lemma}

\begin{proof}
    We only give here some hints on how to prove the result. We can first show
    that $\Delta(u,v,w)$ contains a flat triangle. To do so, set $k :=
    \dist_G(v,w)$ and consider a shortest $(v,w)$-path $P$. We rename its
    vertices by $v =: u_0^k, u_1^k, \ldots, u_k^k := w$, where $u_i^k$ denotes
    the vertex at distance $i$ from $v = u_0^k$ on $P$. By successively
    applying the triangle condition to vertices $u_0^0 := u$, $u_i^k$ and
    $u_{i+1}^k$ for $i \in \{0, \ldots, k-1\}$, we derive vertices $u_i^{k-1}
    \sim u_i^k, u_{i+1}^k$ at distance $k-1$ from $u_0^0$. Continuing so, we
    obtain that $\Delta(u,v,w)$ contains a flat triangle of the form:
    \begin{itemize}
        \item $V_k := \{ u_{i}^{j} : 0 \le j \le k \text{ and } 0 \le i \le j
        \}$;
        \item $E_k := \{ u_{i}^{j}u_{i'}^{j'} : u_{i}^{j},
        u_{i'}^{j'} \in V_k, (j = j' \text{ and } i' = i - 1) \text{ or } (j' =
        j - 1 \text{ and } i - 1 \le i' \le i) \}$.
    \end{itemize}
    It then remains to show that this flat triangle is vertex-maximal. This is
    done by proving that it is locally convex (and thus convex).
\end{proof}

\begin{figure}[htb]
    \begin{minipage}{0.49\linewidth}
        \centering
        \includegraphics[width=0.55\linewidth]{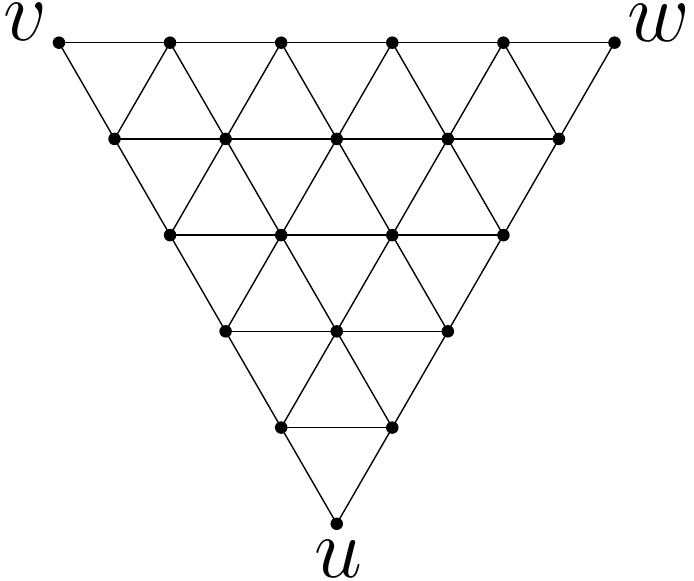}
    \end{minipage}
    \begin{minipage}{0.49\linewidth}
        \centering
        \includegraphics[width=0.85\linewidth]{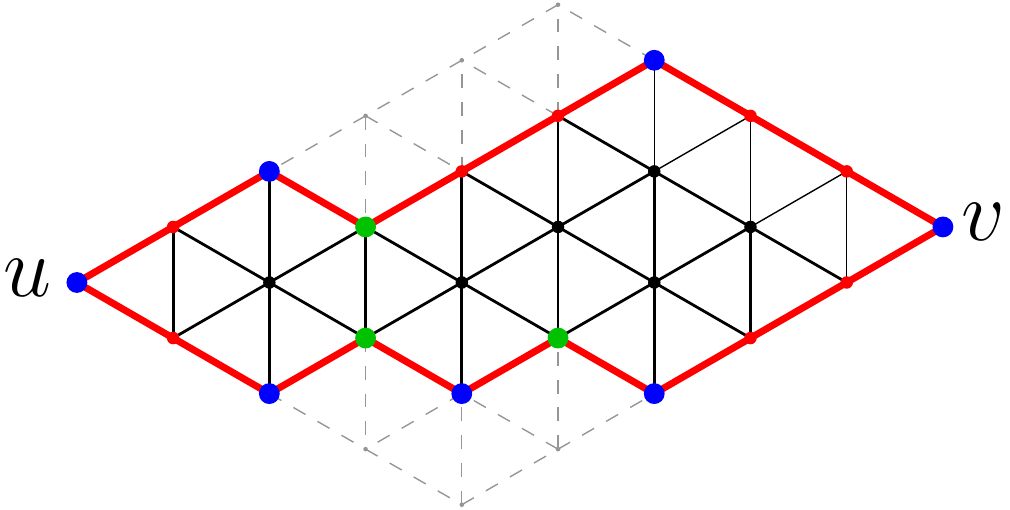}
    \end{minipage}
    \caption{
        \label{fig_flat_triangle}
        \label{fig_grille_rongee}
        A flat triangle (left) and a burned lozenge (right).
        The concave corners are in green and convex ones are drawn in blue. The
        border is in red and the inner vertices in black.
    }
\end{figure}

\begin{figure}[htb]
    \centering
    \includegraphics[width=0.8\linewidth]{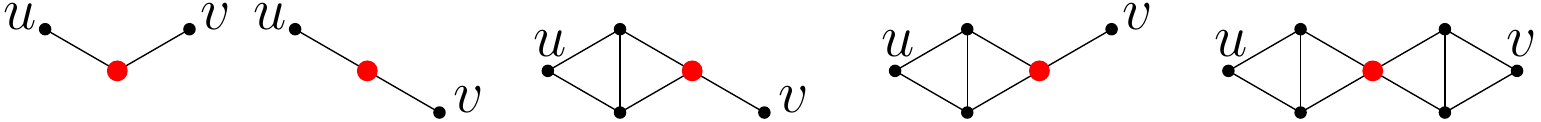}
    \caption{
        \label{fig_degenerated_vertices}
        The five red vertices indicate types of articulation points.
    }
\end{figure}

\begin{lemma}
    \label{lem_interval_shape}
    Any interval $I(u,v)$ of $G$ induces a burned lozenge.
\end{lemma}

\begin{proof}
Let $u$ and $v$ be two vertices at distance $\ell$ of a $K_4$-free bridged graph $G$. Let $C_i := \{w\in I(u,v): d_G(u,w)=i\}$ for $0\leq i\leq \ell$. We first show the following claim.

\begin{claim}\label{claim_burned_lozenge}
For all $0\leq i\leq \ell$, $C_i$ induces a convex path of $G$.
\end{claim}

\begin{proof}
Observe that vertices of $C_i$ are at distance $\ell - i$ from $v$ and so $C_i = B_i(u)\cap B_{\ell-i}(v)$. Since $B_i(u)$ and $B_{\ell-i}(v)$ are both convex, $C_i$ is also convex, and thus $G[C_i]$
is a bridged subgraph. By Lemma \ref{K3}, $G[C_i]$ cannot contains triangles, thus $G[C_i]$ is a tree. Suppose that this tree contains a vertex $x$ with three neighbors $x_1,x_2,x_3$. Let $y_j$ be the common neighbor
of $x$ and $x_j$, $j=1,2,3$, at distance $i-1$ from $u$ (obtained by applying the triangle condition). Since $C_i$ is convex, the vertices $y_1,y_2,y_3$ are pairwise distinct. Since $y_1,y_2,y_3\in B_{i-1}(u)$ and $x\notin B_{i-1}(u)$, the convexity of $B_{i-1}(u)$ implies that $y_1,y_2,y_3$ are pairwise adjacent. Together with $x$ they induce a forbidden $K_4$. This establishes that $C_i$ is a convex path.
\end{proof}

Further, we will suppose that $I(u,v)$ does not contain articulation points, otherwise, we can apply the induction hypothesis to each 2-connected component of $I(u,v)$.
We now show by induction on  $1\le k\le \ell$ that the intersection $I(u,v)\cap B_k(u)$ is a halved burned lozenge in which all spheres $I(u,v)\cap S_{i}(u), 1\le i\le k$ are vertical paths. For this, suppose that $u$ is identified with the origin of the triangular grid. Note that $I(u,v)\cap B_1(u)$ is a triangle $u,w',w''$. 
We embed this triangle in the triangular grid in such a way  that 
$w'w''=I(u,v)\cap S_1(u)$ is a vertical edge to the right of $u$.
Assume by induction hypothesis that the desired embedding property is satisfied for $k-1$, in particular, that $C_{k-1}=I(u,v)\cap S_{k-1}(u)$ is embedded as a vertical path.  
By Claim \ref{claim_burned_lozenge}, $C_{k-1}$ and $C_k$ induce convex paths of 
$G$. From their definition and by Claim \ref{claim_burned_lozenge}, each vertex 
of $C_{k-1}$ is adjacent to at least one and to at most two vertices of $C_k$ 
and, vice versa, each vertex of $C_k$ is adjacent to at least one and to at 
most two vertices of $C_{k-1}$. 
This implies that the lengths of paths $C_k$ and $C_{k-1}$ differ by at most 
1.  
Suppose that $C_{k-1}$ induces the path $P=(u_0,u_1,\dots, u_p)$ for some $p\ge 1$.
For $1\leq i \leq p$, by the triangle condition, there exists a vertex $v_i\in C_k$ at distance $\ell - k$ from $v$ and adjacent to $u_{i-1},u_i$. By Claim~\ref{claim_burned_lozenge}, all $v_i$ are distinct.
For $1\leq i\leq p-1$, $v_i$ and $v_{i+1}$ are adjacent since they are both adjacent to $u_{i+1}$ and 
$C_k$ is convex by Claim~\ref{claim_burned_lozenge}. Hence, the vertices $v_1,\ldots,v_p$ define a subpath $P'$ of $C_k$.
If $C_k=P'$, then one can extends the halved burned lozenge representing $I(u,v)\cap B_{k-1}(u)$ to a one representing $I(u,v)\cap B_{k}(u)$  by embedding the path $P'$ vertically to the right of the path $P$. Now suppose that
$P'\ne C_k$. Since any vertex of $C_k$ is adjacent to at least one vertex of $C_{k-1}$ and each vertex of $C_{k-1}$ is adjacent to at most two vertices of $C_k$, we conclude that $C_k\setminus P'$ may contains either one vertex $x$ or two vertices $x,y$. In the first case, $x$ is adjacent to $v_1$ (or to $v_p$) and to $u_0$ (or to $u_p$). In the second case, $x$ is adjacent to $v_1$ and $u_0$ and $y$ is adjacent to $v_p$ and $u_p$.
In the first case, we add the vertical edge $xv_1$ or $v_px$ to $P'$. In the second edge, we add the vertical edges $xv_1$ and $yv_p$ to $P'$. In both cases, we obtain the representation of $I(u,v)\cap B_k(u)$ as a halved burned lozenge.
If $k=\ell$, then this halved burned lozenge will be a burned lozenge.
\end{proof}

\begin{lemma}
    \label{lem_strongly_equilateral_qm}
    If $uvx'$ is the quasi-median of the triplet $uvx$, and if $w$ is a
    neighbor of $u$ in $I(u,v)$, then $\dist_G(w,x) = \dist_G(u,x) =
    \dist_G(v,x)$.
\end{lemma}

\begin{proof}
    From the definition of the quasi-median, and since $uvx'$ is a strongly
    equilateral metric triangle, we deduce that $\dist_G(u,x) = \dist_G(v,x) =:
    k$ and $\dist_G(w,x) \le k$.
    Suppose by way of contradiction that $\dist_G(w,x) < k$, i.e.,
    $\dist_G(w,x) = k - 1$.
    Consider the following vertices of the deltoid $\Delta(u,v,x')$:
    $u_1$ the common neighbor of $u$ and $w$; $w_1$ the common neighbor of $w$
    and $u_1$; and $u_2$ the common neighbor of $u_1$ and $w_1$, see Figure
    \ref{fig_strongly_equilateral_qm}.
    Since $\dist_G(w,x) = \dist_G(u_1,x) = k - 1$, by triangle condition, there
    exists a vertex $t \sim u_1, w$ having distance $k-2$ to $x$.
    Since $w$ is not adjacent to $u_2$, because of Lemma
    \ref{lem_MT_structure_K4-free}, $t$ is different from $u_2$.
    By convexity of the ball $B_{k-2}(x)$, we deduce that $t \sim u_2$.
    The four-cycle $\{w, t, u_2, w_1\}$ can not be induced. Since $w \not\sim
    u_2$, we conclude that $w_1 \sim t$. Consequently, the vertices $w$, $t$,
    $u_2$, $w_1$ induce a forbidden $K_4$.
\end{proof}

\begin{figure}[htb]
    \centering
    \includegraphics[width=0.23\linewidth]{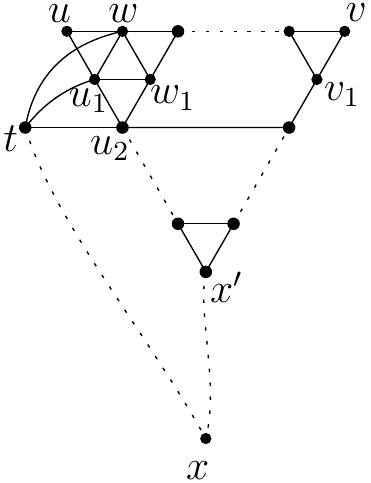}
    \caption{
        \label{fig_strongly_equilateral_qm}
        Illustration of Lemma \ref{lem_strongly_equilateral_qm}.
    }
\end{figure}

\subsection{Starshaped trees}
We now introduce starshaped sets and trees, and we describe the structure of
the intersection of an interval and a starshaped tree.
Let $T$ be a tree rooted at a vertex $z$. A path $P$ of $T$ is called
\emph{increasing} if it is entirely contained on a single branch of $T$, i.e.,
if $\forall u, v \in P$, either $I_T(u,z) \subseteq I_T(v,z)$, or $I_T(v,z)
\subseteq I_T(u,z)$.
Equivalently, an increasing path is the shortest path between two vertices that
are in ancestor-descendant relation.
A subset $S$ of the vertices of a graph $G$ is said to be
\emph{starshaped} relatively to a vertex $z \in S$ if $I(z,s)
\subseteq S$ for all $s \in S$.
If, additionally, every $I(z,s)$ induces a single path of $G$, then
$S$ is called a \emph{starshaped tree} (rooted at $z \in S$), and each interval
$I(z,s)$ is called a \emph{branch} of $S$. Taking the union of all branches, a
starshaped tree $S$ rooted at $z$ is a shortest-path spanning tree of $G[S]$
rooted at $z$.
Similarly to other shortest-path trees (e.g., BFS trees), starshaped trees
are not necessarily induced subgraphs of $G$, see Figure
\ref{fig_total_boundary}.

Let $I(u,z)$ be a burned lozenge which is not a single shortest path of $G$.
Then $I(u,z)$ contains a non-trivial block, i.e., a $2$-connected component of
$G$ containing a triangle.
Let $B$ be such a non-trivial block closest to $z$. Let $z_0$ be the vertex of
$B$ the closest to $z$ and let $u_0$ be the vertex of $B$ the closest to $u$.
The boundary of $B$ defines two shortest $(z_0, u_0)$-paths $Q_1$ and $Q_2$.
Let $v_1$ (resp. $v_2$) be the closest to $z$ convex corner of $I(u,z)$
belonging to $Q_1$ (resp. to $Q_2$).
We call the convex corners $v_1$ and $v_2$ \emph{extremal} relatively to $z$.

\begin{lemma}
    \label{lem_2_increasing_paths}
    Let $T$ be a starshaped tree of $G$ (rooted at $z$), and let $u \in V
    \setminus T$.
    Then $T' := I(u,z) \cap T$ is a starshaped tree.
    Moreover, $T'$ is a tripod consisting of the union of three increasing
    paths $P_0$, $P_1$ and $P_2$ such that (see Figure \ref{fig_tripod}):
    \begin{enumerate}[(i)]
        \item $P_0 = I(z,z_0)$, $P_1 = I(z_0,v_1)$, and $P_2 = I(z_0,v_2)$;
        \item $I(z,v_1) = P_0 \cup P_1$ and $I(z,v_2) = P_0 \cup P_2$;
        \item $z_0$, $v_1$, and $v_2$ are defined with respect to the
        non-trivial block $B$ closest to $z$.
    \end{enumerate}
\end{lemma}

\begin{proof}
    Since $T' \subseteq T$, $T'$ is a forest.
    Pick two vertices $v$ and $w$ of $T'$.
    Then $v \in I(u,z)$ implies that $I(v,z) \subseteq I(u,z)$. Since $v$
    belongs to $T$ and $T$  is starshaped, necessarily $I(v,z) \subseteq T$.
    It follows that $I(v,z) \subseteq T'$ and $I(w,z) \subseteq T'$.
    Consequently, $T'$ is a starshaped tree.

    To prove the second assertion, first notice that, since $u \notin T'$,
    $I(u,z)$ is not a single shortest path. This means that the graph induced
    by $I(u,z)$ contains non-trivial blocks. Therefore the vertices $z_0$,
    $v_1$, $v_2$ and the paths $P_0$, $P_1$, $P_2$ are well defined.
    Moreover, $P_0 \cup P_1 \cup P_2 \subseteq T'$ as a consequence of Lemma
    \ref{lem_interval_shape} and the definition of the paths $P_0, P_1$ and
    $P_2$.
    To prove the converse inclusion $T' \subseteq P_0 \cup P_1 \cup P_2$,
    suppose by way of contradiction that there exists a vertex $v \in T'
    \setminus (P_0 \cup P_1 \cup P_2)$.
    Then $I(v,z)$ is an increasing path. From the structure of $I(u,z)$
    given by Lemma \ref{lem_interval_shape} and from the choice of $v \in T'
    \setminus (P_0 \cup P_1 \cup P_2)$, we conclude that $v$ is a regular
    border or a convex corner of the block $B$.
    Suppose without loss of generality that $v$ belongs to the same border path
    as $v_1$ (see Figure \ref{fig_tripod}). Since $v \ne v_1$, $I(v,z) \cap B$
    is non-trivial, and thus $I(v,z)$ is not a single path, contrary to the
    assumption that $v \in T'$.
    This establishes the converse inclusion, and thus $T' = P_0 \cup P_1 \cup
    P_2$.
\end{proof}

\begin{figure}[htb]
    \centering
    \includegraphics[width=0.7\linewidth]{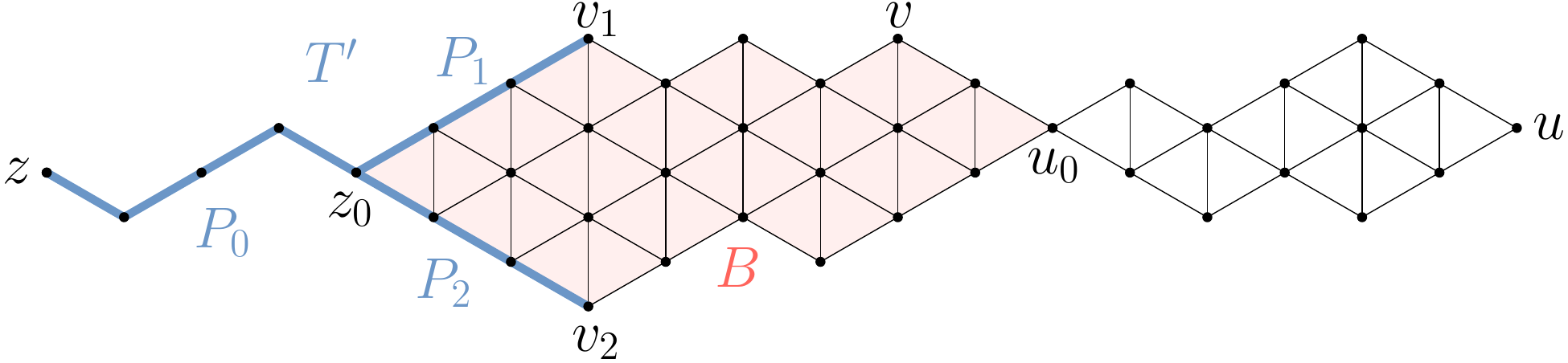}
    \caption{
        \label{fig_tripod}
        Illustration of Lemma \ref{lem_2_increasing_paths}.
    }
\end{figure}

\section{Stars and fibers}
\label{sect_stars_in_dim2_bridged}

\subsection{Projections and stars}
Let $z \in V$ be an arbitrary vertex of $G$. Since $G$ is bridged, $N[z]$ is convex.
Note that for any $u \in V \setminus N[z]$, $z$ cannot belong to
the metric projection $\mproj{u}{N[z]}$. Indeed, $z$ necessarily has a
neighbor $z'$ on a shortest $(u,z)$-path. This $z'$ is closer to $u$ than $z$,
and it belongs to $N[z]$. The projections on $N[z]$ have the following property.

\begin{lemma}
    \label{lem_mproj_on_N[z]} $\mproj{u}{N[z]}$ consists of a
    single vertex or of two adjacent vertices. Moreover, if $d_G(u,z)=k+1$ and $\mproj{u}{N[z]}=\{y,y'\}$, then there exists a unique vertex $x$ adjacent to $y$ and $y'$ and having distance $k$ to $u$.
\end{lemma}

\begin{proof}
Notice first that $z\notin\mproj{u}{N[z]}$ because any neighbor of $z$ in $I(z,u)$ is closer to $u$ than $z$. Suppose that $\mproj{u}{N[z]}$ contains two distinct vertices $y$ and $y'$. Since $y$ and $y'$ are different from $z$, they have distance $k$ to $u$. By convexity of the ball $B_k(u)$, we conclude that $y$ and $y'$ are adjacent. If $\mproj{u}{N[z]}$ contains a third vertex $y''$, then $y,y',y'',z$ induce a forbidden $K_4$.

So, let $\mproj{u}{N[z]}=\{y,y'\}$. Then $d_G(u,y)= d_G(u,y')=k$. Since $y\sim y'$, by triangle condition, there exists a vertex $x\sim y, y'$ and having distance $k-1$ to $u$. If there exists another such vertex $x'$, since $x$ and $x'$ belong to the ball $B_{k-1}(u)$ and are adjacent $y$ and $y'$, there must be adjacent because $B_{k-1}(u)$ is convex. Consequently, the vertices $x,x',y,y'$ induce a forbidden $K_4$.
\end{proof}

Let $u \in V$ be a vertex with two vertices $y,y'$ in  $\mproj{u}{N[z]}$.
By Lemma \ref{lem_mproj_on_N[z]}, there exists
a vertex $u'\sim y,y'$ at distance $\dist_G(y,u) - 1$ from $u$. Moreover,
$I(u',z) = \{u',z,y,y'\}$ and $y \sim y'$.
The \emph{star} $\St(z)$ of a vertex $z \in V$ consists of the neighborhood $N[z]$ of $z$ plus all $u'\notin N[z]$ having two neighbors $y$ and $y'$ in $N[z]$ (which are necessarily adjacent).
Consequently, $\St(z)$ contains the vertices of $N[z]$ and all $u'$ that can be
derived by
the triangle condition applied to two adjacent vertices $y, y' \in N[z]$ and a
vertex $u \in V$. Figure~\ref{fig_star_dim2_bridged} (1) and (2) presents two
examples of stars in $K_4$-free bridged graphs.

\begin{figure}[htb]
    \centering
    \begin{minipage}{0.49\linewidth}
        \centering
        \includegraphics[width=0.65\linewidth]{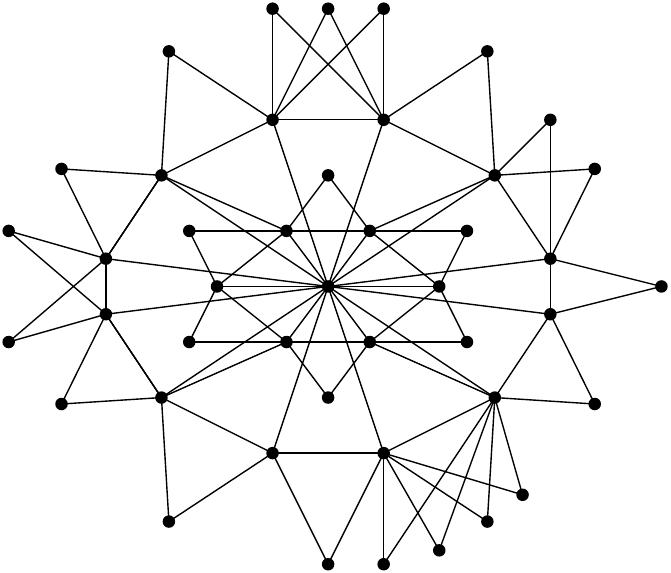}\\
        (1)
    \end{minipage}
    \begin{minipage}{0.49\linewidth}
        \centering
        \includegraphics[width=0.65\linewidth]{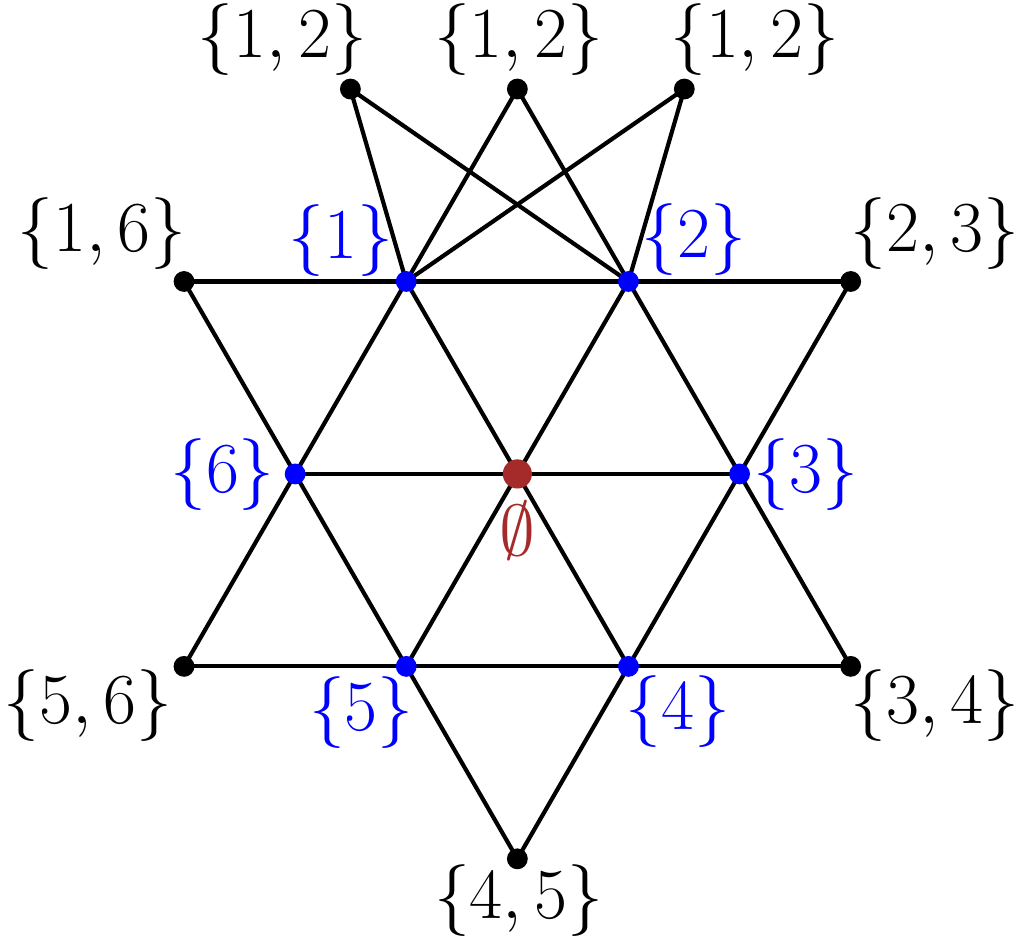}\\
        (2)
    \end{minipage}

    \begin{minipage}{0.49\linewidth}
        \centering
        \includegraphics[width=0.65\linewidth]{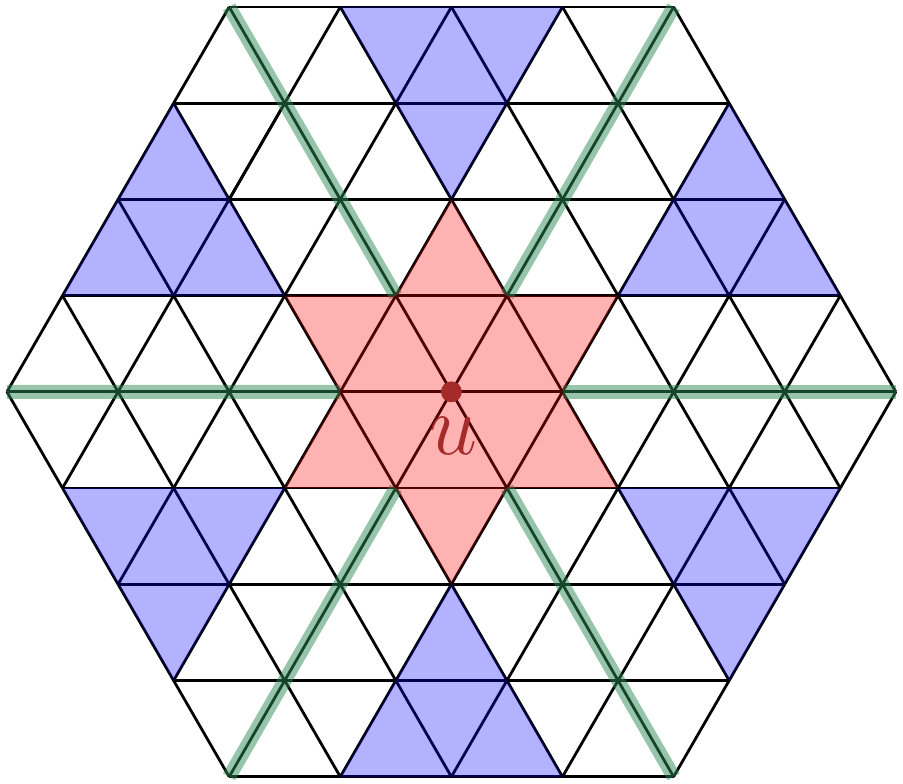}\\
        (3)
    \end{minipage}
    \caption{
        \label{fig_star_dim2_bridged}
        \label{fig_enc_star_bridged}
        Examples (1) and (2) of a star
        and the encoding of the vertices of a star.
        Example (3) of cones (in blue) and panels (in green). The star of $u$ is given in red.
    }
\end{figure}

\subsection{Cones and panels}
Let  $x \in \St(z)$. If $x \in N[z]$, we define the \emph{fiber
$F(x)$ of $x$ with respect to $\St(z)$} as the set of all vertices of $G$
having $x$ as unique projection on $N[z]$ by Lemma~\ref{lem_mproj_on_N[z]}.
Otherwise (if $\dist_G(x,z) = 2$) $F(x)$ denotes the set of all vertices $u$
such that $\mproj{u}{N[z]}$ consists of two adjacent vertices $v$ and $w$, and
such that $x$ is adjacent to $v,w$ and is one step closer to $u$ than $v$ and
$w$.
A fiber $F(x)$ such that $x\sim z$ is called a \emph{panel}. If
$\dist_G(x,z) = 2$,  then $F(x)$ is called a \emph{cone}.
Figure~\ref{fig_star_dim2_bridged} (3) illustrates cones and panels.
Two fibers $F(x)$ and $F(y)$ are called \emph{$k$-neighboring} if $\dist_{\St(z)
\setminus \{z\}}(x, y) = k$.
Notice that any cone is $1$-neighboring exactly two panels.

\begin{lemma} \label{lem_starshaped} Each fiber $F(x), x\in \St(z)$ is starshaped with respect to $x$.
\end{lemma}

\begin{proof} Pick $u \in F(x)$ and $w \in I(u,x)$. Then $\mproj{w}{N[z]} \subseteq \mproj{u}{N[z]}$. If $F(x)$ is a panel then $x$ is the unique projection of $u$ on $N[z]$. Since $w\in I(u, x)$, $x$ is also the unique projection of $w$ on $N[z]$. Thus $w$ belongs to $F(x)$.
If $F(x)$ is a cone, then the projection of $u$ on $N[z]$ consists of two vertices $x_1$, $x_2$ both adjacent to $z$ and $x$. Again, since $w\in I(u,x)$, $x_1$ and $x_2$ are projections of $w$ on $N[z]$, yielding $w\in F(x)$.
\end{proof}

\begin{lemma}\label{lem_c2c_and_p2p}
    Let $u\in F(x),v\in F(y)$ and $F(x)\ne F(y)$. If the fibers $F(x)$ and $F(y)$ are both cones or
    are both panels, then the vertices $u$ and $v$ are not adjacent.
\end{lemma}

\begin{proof}
    Suppose  $u\sim v$. Notice that this implies that $\dist_G(u,x) = \dist_G(v,y) =:k$.
    Indeed, if $\dist_G(u,x) > k$, then $y \in \mproj{u}{\St(z)}$ and if
    $\dist_G(u,x) < k$, then $x \in \mproj{v}{\St(z)}$.

    First, let $F(x)$ and $F(y)$ be two cones.
    If $F(y)$ and $F(y)$ are 1-neighboring $F(x)$, then $x\sim y$ and $G$ will contain a forbidden $C_5, C_4$, or $K_4$.
    If $F(x)$ and $F(y)$ are 2-neighboring, then there exists a vertex $w \in \St(z)$ adjacent to $x$ and $y$ and at distance $k+1$ to $u$ and $v$. Then
    $x,y\in B_k(\{ u,v\})$ and $w \notin B_{k}(\{u,v\})$, contrarily to
    convexity
    of  $B_k(\{ u,v\})$. Thus, the cones $F(x)$ and $F(y)$ are $r$-neighboring
    for some $r > 2$. This implies that $\mproj{u}{N[z]} \cap \mproj{v}{N[z]}=\varnothing$.
    By the triangle condition, there exists a vertex $t \sim u, v$ at distance
    $k + 1$ from $z$. Since $t\in I(u,z) \cap I(v,z)$, we conclude that
    $\mproj{t}{N[z]}\subseteq \mproj{u}{N[z]} \cap \mproj{v}{N[z]}$. This is impossible
    since $\mproj{t}{N[z]}\neq \varnothing$ and $\mproj{u}{N[z]} \cap 
    \mproj{v}{N[z]}=\varnothing$.
    Now, let $F(x)$ and $F(y)$ be two panels. Then $\mproj{u}{N[z]} = \{x\}$ and $\mproj{v}{N[z]} = \{y\}$ and thus $\mproj{u}{N[z]}\cap \mproj{v}{N[z]} = \varnothing$.
    Since $\dist_G(u,x) = \dist_G(v,y) = k$ and $\dist_G(u,z) = \dist_G(v,z) = k+1$,
    $x,y\in B_k(\{u,v\})$ and $z\notin B_k(\{u,v\})$.  Since $u\sim v$, $B_k(\{u,v\})$ is convex, $x$ and $y$ must be adjacent.
    Since $u$ and $v$ are adjacent and are at distance $k + 1$ from $z$, by triangle condition, there exists a vertex $t \sim u, v$ with $\dist_G(t,z) = k$.
    Since $t\in I(u,z) \cap I(v,z)$, we have $\mproj{t}{N[z]}\subseteq 
    \mproj{u}{N[z]} \cap \mproj{v}{N[z]} = \varnothing$, which is impossible.
\end{proof}

\begin{lemma}
    \label{lem_dist_to_base}
    Let  $u \in F(x), v \in F(y)$, and $u\sim v$.
    If $F(y)$ is a cone and  $F(x)$ is a panel, then $x\sim y$ and  $\dist_G(u,y) =
    \dist_G(u,x) \in \{k, k + 1\}$, where  $k := \dist_G(v,y)$.
\end{lemma}

\begin{proof} Since $u$ and $v$ are adjacent, $\dist_G(u,y)\le k+1$.
By Lemma \ref{lem_starshaped} $F(y)$ is starshaped with respect to $y$ and $v\in F(y), u\notin F(y)$, thus
$\dist_G(u,y)\ge k$.  Consequently, $\dist_G(u,y)\in \{k, k + 1\}$.
Let $y_1$ and $y_2$ denote the two neighbors of $y$ in $\St(z)$. Then $\dist_G(v,y_1)=\dist_G(v,y_2)=k+1$
and $\dist_G(v,z)=k+2$.

First suppose that  $x$ coincides with $y_1$ or $y_2$, say $x=y_2$. Therefore,
$x\sim y$ and $\dist_G(v,x)=k+1$. It remains to show that in this case $\dist_G(u,y)=\dist_G(u,x)$. Let $x'$ be a neighbor
of $x$ in $I(x,u)$. If $\dist_G(u,x)=k$, then $y,x'\in B_{k}(v)$ and $x\notin B_k(v)$. By the convexity of $B_k(v)$,
$y$ and $x'$ are adjacent. This implies that $\dist_G(u,y)\le k$. Since $\dist_G(u,y)\ge k$, we conclude that $\dist_G(u,y)=k=\dist(u,x)$.
Now suppose that $\dist_G(u,x)=k+1$. If $\dist_G(u,y)=k$, then this would imply that $u$ must belong to the cone $F(y)$, a contradiction.
This establishes the assertion of the lemma when $x\in \{ y_1,y_2\}$.

We show that $x\in\{y_1, y_2\}$. Assume by contradiction that $x$ is different from $y_1$ and $y_2$. This implies that $\dist_G(v,x)>k+1=\dist_G(v,y_1)=\dist_G(v,y_2)$ and since $v\sim u$,
we conclude that $\dist_G(u,x)\ge k+1$. Analogously, since $\dist_G(u,y_1)\le k+2$ and $\dist_G(u,y_2)\le k+2$ and are both longer than $\dist_G(u,x)$,
we must have $\dist_G(u,x)=k+1$ (and  $\dist_G(u,y_1)=\dist_G(u,y_2)=k+2$). Thus $\dist_G(u,z)=k+2=\dist(v,z)$. Consider the ball $B_{k+1}(\{ u,v\})$ of radius $k+1$ around the convex set
$\{ u,v\}$, which must be convex. Since $y_2,y_2,x\in B_{k+1}(\{ u,v\})$ and $z\notin B_{k+1}(\{ u,v\})$, the convexity of $B_{k+1}(\{ u,v\})$  implies that $x\sim y_1,y_2$.
Consequently, we obtain the forbidden $K_4$ induced by the vertices $y_1,y_2,z,x$ and thus a contradiction.
\end{proof}

\subsection{Partition of $G$ into fibers}

We continue by showing that the fibers of any star of $G$ defines a partition of $G$ into cones and panels.

\begin{lemma}
    \label{lem_bridged_and_starshaped_partitionning}
    $\mcf_z := \{ F(x) : x \in \St(z) \}$ defines a partition of  $G$.
    Any fiber $F(x)$ is a bridged isometric subgraph of $G$ and $F(x)$ is starshaped
    with respect to $x$.
\end{lemma}

\begin{proof}
    The fact that $\mcf_z$ is a partition follows from its definition. Since any isometric subgraph
    of a bridged graph is bridged and each fiber $F(x)$ is starshaped by Lemma \ref{lem_starshaped},
    we have to prove that $F(x)$ is isometric. Let $u$ and $v$ be two
    vertices of $F(x)$. We consider a quasi-median $u'v'x'$ of the triplet $u,v,x$.
    Since $F(x)$ is starshaped, the intervals
    $I(u',x')$, $I(x,x')$, $I(u,u')$, $I(v,v'),$ and $I(v',x')$ are all
    contained in $F(x)$. Consequently, to show that $u$ and $v$ are connected in $F(x)$ by a shortest path,
    it suffices to show that the unique shortest $(u',v')$-path in the deltoid $\Delta(u',v',x')$ belongs to $F(x)$.
    To simplify the notations, we can assume two things.
    First, since $I(u,u'), I(v,v') \subseteq F(x)$, we can let $u = u'$ and $v = v'$.
    Second, we can assume that $\Delta(u,v,x')$ is a minimal counterexample 
    with $I(u,v)\nsubseteq F(x)$.

    Let $w$ be the vertex  closest to $u$ on the $(u,v)$-shortest path of $\Delta(u,v,x')$
    such that $w \notin F(x)$. Then we can suppose that
    $u$ is adjacent to $w$, otherwise, we can replace $u$ by the
    neighbor $u'$ of $w$ in $F(x)\cap \Delta(u,v,x')$ and obtain a smaller counterexample $\Delta(u',v,x')$.
    By Lemma \ref{lem_strongly_equilateral_qm}, we conclude that $\dist_G(x,u)
    = \dist_G(x,w) = \dist_G(x,v) = k$.
    By Lemma~\ref{lem_MT_structure_K4-free}, there exist a vertex $u_1 \sim w,
    u$, a vertex $w_1 \sim u_1, w$, and a vertex $v_1 \sim v$ such that $w_1
    \in I(u_1, v_1)$.
    Applying Lemma \ref{lem_strongly_equilateral_qm} once again, we deduce that
    $\dist_G(x,u_1) = \dist_G(x,w_1) = \dist_G(x,v_1) = k - 1$.
    By the minimality choice of the counterexample, we also deduce that the
    vertices $u_1, v_1$ and $w_1$ belong to $F(x)$.
    Moreover, the deltoid $\Delta(u,v,x')$ is entirely contained in $F(x)$ (see
    Fig. \ref{fig_isometric_fibers}).
    Two cases have to be considered.

    \begin{figure}[t]
    \centering
    \includegraphics[width=0.23\linewidth]{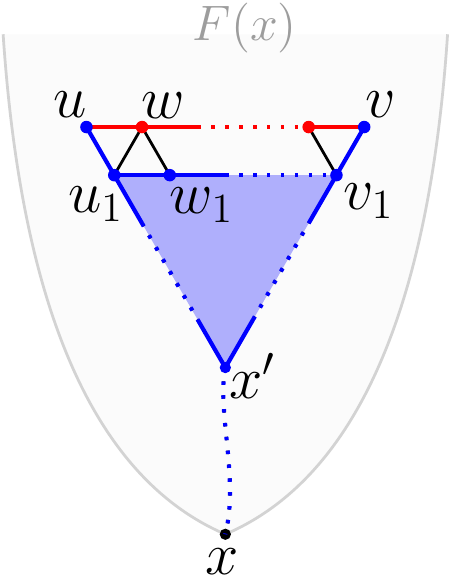}
    \caption{
        \label{fig_isometric_fibers}
        Illustration of the proof of Lemma
        \ref{lem_bridged_and_starshaped_partitionning}.
        The minimality hypothesis implies that the blue part belongs to the fiber
        $F(x)$. The proof aims to show that the red shortest $(u,v)$-path
        also belongs to $F(x)$.
    }
\end{figure}

    \smallskip\noindent
    \textbf{Case 1.} $F(x)$ is a panel.
    Then, by Lemma \ref{lem_c2c_and_p2p}, $w$ belongs to a cone $F(y)$
    1-neighboring
    $F(x)$. Since $\dist_G(w,x) = k$, $\dist_G(w,y) = k - 1$.
    Moreover, by Lemma \ref{lem_dist_to_base}, we know that $\dist_G(w_1,y) =
    \dist_G(w_1,x) = k - 1$.
    By triangle condition applied to  $w$, $w_1$, and $y$, there exists a
    vertex $t \sim w, w_1$ at distance $k - 2$ from $y$.
    Since $F(y)$ is starshaped  and  $t \in I(w,y)$, $t \in F(y)$.
    Also, since $F(y)$ is a cone, we obtain that $\dist_G(t,x) = k - 1$.
    By the convexity of $B_{k-1}(x)$, $t$ must coincide with $u_1$ or with
    $w_1$ (otherwise, the quadruplet $u_1,w_1,w,t$ would induce a $K_4$).
    Since $u_1$ and $w_1$ belong to $F(x)$,
    then $t \in F(x)$, leading  to a contradiction.

    \smallskip\noindent
    \textbf{Case 2.} $F(x)$ is a cone.
    Then $w$ belongs to a panel $F(y)$ 1-neighboring $F(x)$, and this case is
    quite
    similar to the previous one. By Lemma \ref{lem_dist_to_base}, $\dist_G(w,y)
    = \dist_G(w,x) = k$.
    Recall that $\dist_G(w_1,x) = k - 1$, $w_1 \in F(x)$ and $F(x)$ is a cone.
    Consequently $\dist_G(w_1,y) = k$.
    We thus have $\dist_G(w,y) = \dist_G(w_1,y) = k$ and, by triangle
    condition, there exists a vertex $t \sim w, w_1$ at distance $k - 1$ from
    $y$.
    Still using Lemma \ref{lem_dist_to_base}, we deduce that $\dist_G(t,x) = k
    - 1$.
    Since $F(y)$ is starshaped, we obtain that $t \in F(y)$, and from the
    convexity of the ball $B_{k-1}(x)$ we conclude that $t = w_1$ or $t = u_1$.
    Finally, $t \in F(x)$.
\end{proof}

If we choose  the star centered at a median vertex of $G$, then the number of vertices in each fiber is bounded by $|V|/2$ (the proof is similar to the proof of \cite[Lemma 10]{ChLaRa_labeling_median}). For an edge $uv\in E(G)$, let $W(u,v) := \{w : d_G(u,w)<d_G(v,w)\}$.

\begin{lemma}
    \label{lem_small_fibers_dim2_bridged}
    If $z$ is a median vertex of $G$, then for all $x \in \St(z)$, $|F(x)| \le
    |V| / 2$.
\end{lemma}

\begin{proof}
    Suppose by way of contradiction that $|F(x)| > n/2$ for some vertex $x \in
    \St(z)$. Let $u$ be a neighbor of $z$ in $I(x,z)$.
    If $v\in F(x)$, then  $x\in I(v,z)$ and  $u\in I(x,z)$, and we conclude
    that $u\in I(v,z)$. Consequently, $F(x)\subseteq W(u,z)$, whence
    $|W(u,z)|>n/2$. Therefore $|W(z,u)|=n-|W(u,z)|<n/2$. But this contradicts
    the fact that $z$ is a median of $G$. Indeed, since $u\sim z$, one can
    easily show that $M(u)-M(z)=|W(c,z)|-|W(u,z)|<0$.
\end{proof}

\section{Boundaries and total boundaries of fibers}
\label{sect_boundaries}

\subsection{Starshapeness of total boundaries}
Let $x$ and $y$ be two vertices of $\St(z)$.
The \emph{boundary $\partial_{F(y)} F(x)$ of $F(x)$ with respect to $F(y)$}
is the set of all vertices of $F(x)$ having a neighbor in $F(y)$. The
\emph{total boundary} $\partial^* F(x)$ of $F(x)$ is the union of all its
boundaries (see Fig. \ref{fig_total_boundary}).

\begin{figure}[htb]
    \centering
    \includegraphics[width=0.4\linewidth]{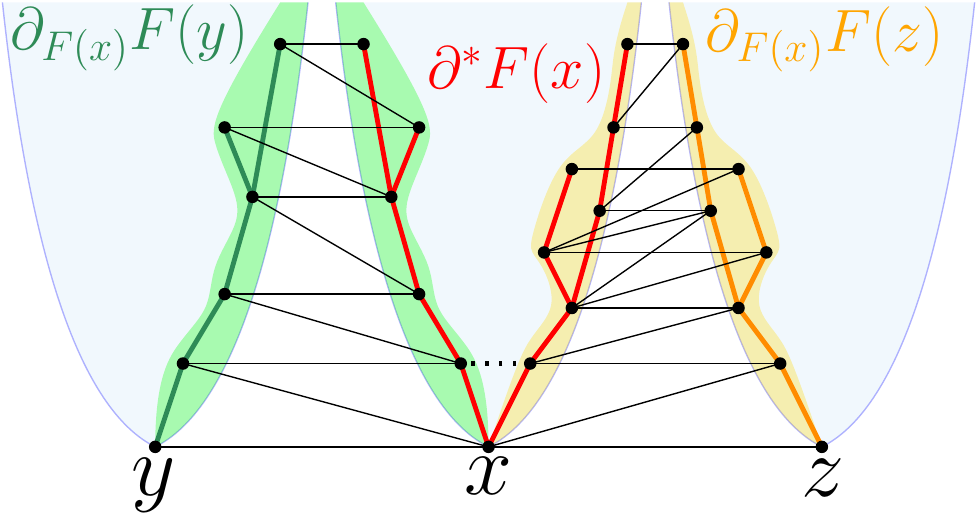}
    \caption{
        \label{fig_total_boundary}
        The boundaries $\partial_{F(x)} F(y)$ and
        $\partial_{F(x)} F(z)$ and of the total boundary $\partial^* F(x)$.
        The edges of this starshaped tree are indicated in red. The black dotted edge links two vertices of the starshaped tree but is not in the tree.
    }
\end{figure}

\begin{lemma}
    \label{lem_total_boundary_starshaped_tree}
    The total boundary $\partial^* F(x)$ of any fiber $F(x)$ is a
    starshaped tree.
\end{lemma}

\begin{proof} To show that $\partial^* F(x)$ is a starshaped tree, it suffices to show
    that (1) for every $v \in \partial^* F(x)$ the interval $I(v,x)$ is
    contained in $\partial^* F(x)$, and (2) $v$ has a unique neighbor in
    $I(v,x)$. By Lemma \ref{lem_starshaped}, $F(x)$ is starshaped, thus $I(v,x)$ is contained in $F(x)$.

    First let $F(x)$ be a cone. Then $x$ has distance $2$ to $z$, and $x$ and $z$ have exactly
    two common neighbors $y$ and $y'$.
    Let $u$ be a neighbor of $v$ in a fiber 1-neighboring $F(x)$. By
    Lemma \ref{lem_c2c_and_p2p}, $u$ necessarily belongs to a panel $F(y)$ or
    $F(y')$, say $u \in F(y)$. Also, we assume that $u$ is a closest to $y$ neighbor of $v$ in
    $F(y)$. Let $k := \dist_G(v,x)$.
    By the definition of a cone, we have $\dist_G(v,y) = \dist_G(v,y') = k + 1$
    and $\dist_G(v,z) = k + 2$. This implies that $\dist_G(u,y) \ge k$.
    Since $u$ belongs to the panel $F(y)$, $\dist_G(u,x) \ge \dist_G(u,y)$.
    Let $u'$ be an arbitrary neighbor of $u$ in $I(u,y)$.
    If $\dist_G(u,y) = k + 1$, from $\dist_G(u',y) = \dist_G(v,x) = k$,
    $\dist_G(u,x) \ge \dist_G(u,y) = k + 1$, and from the convexity of
    $B_k(\{x,y\})$, we conclude that $u' \sim v$. This contradicts the choice
    of $u$.
    So $\dist_G(u,y) = k$.
    Pick any neighbor $w$ of $v$ in $I(v,x)\subset F(x)$. We assert that $w\sim u$. This would
    imply that $w\in \partial^* F(x)$ and since $G$ is $K_4$-free that $v$ has a unique neighbor in $I(v,x)$.
    Indeed,  $\dist_G(y,w) = \dist_G(y,u) = k$ and $\dist_G(y,v) = k + 1$.
    From the convexity of the ball  $B_k(y)$, we obtain that $w \sim u$. 
    This establishes that $I(v,x)$ is a  path included in $\partial^* F(x)$.
    Notice also that the unique neighbor $w$ of $v$ in $I(v,x)$ must be
    adjacent to every neighbor $u'$ of $u$ in $I(u,y)$ because $u', w \in
    B_{k-1}(\{x,y\})$.
    Indeed, since $u \sim u', w$ and $\dist_G(u,x) = \dist_G(u,y) = k$, from
    the convexity of $B_{k-1}(\{x,y\})$ we conclude that $u' \sim w$.

    Now let $F(x)$ be a panel  and pick any vertex  $u \in \partial^* F(x)$.
    As in previous case, we have to show that $I(u,x) \subseteq \partial^* F(x)$ and that $u$
    has a unique neighbor in $I(u,x)$.
    Let $v$ be a neighbor of $u$ in a fiber $F(y)$ 1-neighboring $F(x)$. By Lemma \ref{lem_c2c_and_p2p},
    $F(y)$ is a cone such that $y \sim x, x'$, with $x' \sim x$ and $z\sim x, x'$.
    Assume that $v$ is a closest to $y$ neighbor of $u$ in $F(y)$.
    Let $\dist_G(v,y) := k$.
    By Lemma \ref{lem_dist_to_base}, $\dist_G(u,x) \in \{k, k +
    1\}$. If $\dist_G(u,x) = k$, then we deduce that $u$ is adjacent to the
    neighbor $w$ of $v$ in $I(v,y)$, contrary to the choice of $v$.
    Thus $\dist_G(u,x) = k + 1$.
    In that case, if $u'$ denotes a neighbor of $u$ in $I(u,x)$, then
    $\dist_G(u',x) = \dist_G(v,y) = k$ and $\dist_G(u,y) \ge \dist_G(u,x) = k +
    1$.
    From the convexity of $B_k(\{y,x\})$, we conclude that $u' \sim v$. This
    implies that, if $u$ has two neighbors $u'$ and $u''$ in $I(u,x)$, then
    $u$, $u'$, $u''$, and $v$ induce a forbidden  $K_4$. 
    Consequently, $I(u,x)$ is a path included in $\partial^* F(x)$.
\end{proof}

Total boundaries of  fibers are starshaped trees.
The following result is a corollary of Lemma
\ref{lem_total_boundary_starshaped_tree}.

\begin{corollary}
    \label{corol_mult2_error_total_boundary}
    Let $x$ be an arbitrary vertex of $\St(z)$. Then, for every pair $u,v$ of
    vertices of $\partial^* F(x)$, $\dist_G(u,v) \le \dist_{\partial^*
    F(x)}(u,v) \le 2 \cdot \dist_G(u,v)$.
\end{corollary}

\begin{proof}
    Let $(u',v',x')$ be a quasi-median of $(u,v,x)$.
    Then $\dist_G(u,v) = \dist_G(u,u') + \dist_G(u',v') + \dist_G(v',v)$.
    Since ${\partial^*
    F(x)}$ is a starshaped tree by Lemma
    \ref{lem_total_boundary_starshaped_tree}, $\dist_{\partial^*
    F(x)}(u,u') = \dist_G(u,u')$, $\dist_{\partial^*
    F(x)}(v,v') = \dist_G(v,v')$, $\dist_{\partial^*
    F(x)}(u',x') = \dist_G(u',x')$,
    and $\dist_{\partial^*
    F(x)}(x',v') = \dist_G(x',v')$.
    Moreover, $x' \in I(u,x) \cap I(v,x)$ implies that $x'$ is the nearest
    common ancestor of $u$ and $v$ in ${\partial^*
    F(x)}$.
    Since metric triangles are equilateral in bridged graphs, $\dist_{\partial^*
    F(x)}(u',x') +
    \dist_{\partial^*
    F(x)}(x',v') = 2 \dist_G(u',v')$, yielding the required inequality.
\end{proof}

\subsection{Projections on total boundaries}
We now describe the structure of metric projections of the vertices on the
total boundaries of fibers. Then in Lemma
\ref{lem_exits} we prove that vertices in panels have a constant number of
``exits'' on their total boundaries, even if the panel itself may have an arbitrary
number of 1-neighboring cones.

\begin{lemma}
    \label{lem_proj_increasing_tree}
    Let $F(x)$ be a fiber and $u \in V \setminus F(x)$. Then the metric
    projection $\Pi := \mproj{u}{F(x)} = \mproj{u}{\partial^* F(x)}$ is an
    induced tree of $G$.
\end{lemma}

\begin{proof}
    The metric projection of $u$ on $F(x)$ necessarily belongs to a boundary,
    that is a starshaped tree by Lemma
    \ref{lem_total_boundary_starshaped_tree}. As a consequence, $\Pi$ is a
    \emph{starshaped forest}, i.e., a set of starshaped trees.
    We assert that in fact $\Pi$ is a connected subgraph of $T := \partial^* F(x)$.
    To prove this, we will prove the stronger property that $\Pi$ is an induced 
    tree of $G$.
    Assume by way of contradiction that two vertices $v,w$ of $\Pi$ are not
    connected in $T$ by a path.
    First suppose that $w$ is an ancestor of $v$ in $T$.
    Then every vertex $z$ on the branch of $v$ between $v$ and $w$ belongs to a
    shortest $(v,w)$-path (because $T$ is starshaped) and is at distance at
    most $k:=\dist_G(u,v)= \dist_G(u,w)$ from $u$ (because the ball $B_k(u)$ is convex).
    But then we conclude that the whole shortest $(v,w)$-path belongs to $\Pi$,
    contrary to the choice of $v,w$. Therefore, further we can suppose that $v$
    and $w$ belong to distinct branches of $T$.

    Let $t$ be the nearest common ancestor of $v$ and $w$ in $T$. Let $v'w't'$
    be a quasi-median of $v$, $w$, and $t$. Since $T$ is a starshaped tree, $t'
    \in I(v,t) \cap I(w,t)$, and $t$ is the nearest common ancestor of $v$ and
    $w$, we conclude that $t = t'$.
    We assert that the set of all vertices of $T$ between $v$ and $v'$, between
    $v'$ and $w'$, and between $w'$ and $w$ belongs to $\Pi$.
    Indeed, such vertices belong to a shortest $(v,w)$-path. The ball $B_k(u)$
    is convex and $v, w \in B_k(u)$. Since all such vertices also belong to $T$,
    we conclude that they all have distance $k$ from $u$. We now show that $v' = w' = t$.
    Assume by way of contradiction that $v' \ne w'$ and consider an edge $ab$
    on the path between $v'$ and $w'$ in $\Delta(v',w',t)$.
    According to Proposition \ref{metric_triangles}, $\dist_G(a,x) =
    \dist_G(b,x) =: \ell$.
    By the triangle condition applied to $a$, $b$ and $x$, there exists a
    vertex $c \sim a,b$ at distance $\ell-1$ from $x$.
    Since $a \sim b$ and $\dist_G(a,u) = \dist_G(b,u) = k$, there must exist a
    vertex $d \sim a,b$ at distance $k - 1$ from $u$ and this vertex belongs
    to a fiber $F(y)$ 1-neighboring $F(x)$. By Lemma \ref{lem_c2c_and_p2p}, one
    of those two fibers has to be a panel and the other must be a cone.

    First, let $F(x)$ be a panel and $F(y)$ be a cone.
    By Lemma \ref{lem_dist_to_base}, two subcases have to be considered :
    $\dist_G(d,y) = \ell$  and $\dist_G(d,y) = \ell - 1$. If $\dist_G(d,y) =
    \ell$, then $\dist_G(d,y) = \dist_G(y,c) = \ell$ and
    $\dist_G(b,y) = \ell + 1$. The convexity of $B_{\ell}(y)$ implies that $c
    \sim d$, and therefore $\{a,b,c,d\}$ induce a forbidden $K_4$. If
    $\dist_G(d,y) = \ell - 1$, consider the vertex $x' \sim y, x$ in $I(z,y)$.
    Then $\dist_G(d,x') = \dist_G(c,x') = \ell$, but $\dist_G(b,x') = \ell +
    1$. From the convexity of $B_{\ell}(x')$, we deduce that $c \sim d$.
    So $\{a,b,c,d\}$ induces a forbidden $K_4$.

    Now, let $F(x)$ be a cone and $F(y)$ be a panel. By Lemma
    \ref{lem_dist_to_base}, we have to consider the subcases $\dist_G(d,y) =
    \ell$ and $\dist_G(d,y) = \ell + 1$. If $\dist_G(d,y) = \ell$, then
    $\dist_G(d,y) = \dist_G(c,y) = \ell$ and $\dist_G(b,y) = \ell + 1$ lead to
    $c \sim d$ by the convexity of $B_{\ell}(y)$. Consequently, $\{a,b,c,d\}$
    induces a $K_4$. Finally, if $\dist_G(d,y) = \ell + 1$, we consider a
    vertex $e \sim d$ in $F(y)$ on a shortest $(d,y)$-path, and we consider the
    vertex $y' \sim x, y$ in $I(x,z)$. Then $\dist_G(b,y') = \dist_G(e,y') =
    \ell + 1$ and $\dist_G(d,y') = \ell + 2$. From the convexity of
    $B_{\ell+1}(y')$, it follows that $e \sim b$. With similar arguments, we
    show that $a \sim e$.
    Consequently, $\{a,b,d,e\}$ induces a forbidden $K_4$. Summarizing, we
    showed that in all cases the assumption $v'\ne w'$ leads to a
    contradiction. Therefore  $v' = w' = t$ and $t \in \Pi$. Since $t$ is an
    ancestor of $v$ and $w$, by what has been shown above, $v$ and $w$ can be
    connected in $\Pi$ to $t$ by shortest paths.  This leads to a contradiction
    with the assumption that $v$ and $w$ are not connected in $\Pi$.
\end{proof}

\begin{lemma}
    \label{lem_entrance}
    Let $F(x)$ be a fiber and $u \in V \setminus F(x)$.
    There exists a unique vertex $u'\in \Pi := \mproj{u}{F(x)}$ that is closest
    to $x$.
    Furthermore, $\dist_{G(\Pi)}(u',v) \le \dist_G(u,u') = \dist_G(u,v)$ for
    all $v \in \Pi$.
\end{lemma}

\begin{proof}
    The uniqueness of $u'$ follows from the fact that $\Pi$ is a rooted
    starshaped subtree of $T := \partial^* F(x)$, which itself is a starshaped
    tree rooted at $x$. Indeed, every pair $(a,b)$ of vertices  of $\Pi$ admits
    a nearest common
    ancestor in $\Pi$ that coincides with the nearest common ancestor in $T$.
    This ancestor has to be closer to $x$ than $a$ and $b$ (or at equal
    distance if $a = x$ or $b = x$).

    Pick $v \in \Pi$. The equality $k := \dist_G(u,u') =
    \dist_G(u,v)$ holds since $\Pi$ is the metric projection of $u$
    on $F(x)$. Assume by way of contradiction that $\dist_G(u',v) \ge k + 1$. Then
    $I(v,u') =: P$ is an increasing path of length at least $k + 1$ in a
    starshaped tree.
    By  triangle condition applied to $u$ and to every pair of neighboring
    vertices of $P$, we derive at least $k$ vertices. We then can show that
    each of these (at least $k$) vertices has to be distinct from every other
    (otherwise, $P$ would contain a shortcut).
    We also can prove that those vertices create a path and, by induction on
    the length of this new shortest path, deduce that $(u,v,t)$ forms a
    non-equilateral metric triangle, which is impossible.
\end{proof}

\subsection{The distance lemma}
Lemma \ref{lem_entrance} establishes that every branch of
the tree $\Pi$ has depth smaller or equal to $\dist_G(u,u')$.
The vertex $u'$ defined in Lemma~\ref{lem_entrance} will be called the \emph{entrance} of vertex $u$ in the fiber $F(x)$.

\begin{lemma}
    \label{lem_exits}
    Let $u$ be any vertex of $G$ and let $T$ be a starshaped tree rooted at $z
    \in V$.
    Let $u_1$ and $u_2$ be the two extremal vertices with respect to $z$ in the
    two increasing paths of $I(u,z) \cap T$ (there are at most two of them by
    Lemma \ref{lem_2_increasing_paths}). Then, for all $v \in T$, the following
    inequality holds
    $$
    \min \{
        \dist_G(u,u_1) + \dist_T(u_1,v),
        \dist_G(u,u_2) + \dist_T(u_2,v)
    \} \le 2 \cdot \dist_G(u,v).
    $$
\end{lemma}

\begin{proof}
    Assume $\min_{i \in \{1,2\}} \{\dist_G(u,u_i) +
    \dist_T(u_i,v) \}$ is reached  for $i = 1$.
    Let $x \in T$ be the nearest common ancestor of $u_1$ and $v$. By Lemma \ref{lem_2_increasing_paths},
    we know that $v \notin I(u,z)$, unless $v = x$. We can assume that $v \ne x$, otherwise $\dist_G(u,v) =
    \dist_G(u,u_1) + \dist_T(u_1, v)$ would be shown already.
    Consider a quasi-median $u'v'z'$ of the triplet $u,v,z$ (see Fig.
    \ref{fig_2_representants_bordure}, left).
    We can make the following two remarks:

    (1) Since $T$ is a starshaped tree and $I(z,v)$ is one of its branches,
    $v'$ necessarily belongs to $I(z,v)$. If $v' \in I(z,x)$, then
    $v' \in I(u_1,z) \cap I(v,z)$ implies that $v' =
    x$ and that $\dist_G(u,u_1) + \dist_T(u_1,v') + \dist_T(v',v) =
    \dist_G(u,v)$. Indeed, if $v' = x$, then $v' \in I(u,z)$ (because $x \in I(u,z)$). So
    $\dist_G(u,u_1) + \dist_G(u_1,v') = \dist_G(u,v')$. Since $T$ is
    starshaped, $\dist_G(u_1,v') = \dist_T(u_1,v')$ and $\dist_T(v',v) =
    \dist_G(v',v)$. Finally, since $v'$ belongs to a quasi-median of $u,v,z$, $v'$ lies
    on a shortest $(u,v)$-path. So $\dist_G(u,u_1) +
    \dist_T(u_1,v') + \dist_T(v',v) = \dist_G(u,v') + \dist_G(v',v) =
    \dist_G(u,v)$. Therefore, $v' \in I(x,v)$.

    (2) Since $T$ is starshaped and $z' \in I(v,z)$, $z'$ belongs to the
    branch $I(v,z)$ of $T$. Since $z' \in I(u,z)$ $I(x,v) \cap
    I(u,z) = \{x\}$, we conclude that $z'$ is between $z$ and $x$.
    Since $z' \in I(u',z) \cap I(u_1,z)$, and $u_1, u' \in I(u,z)$, we assert
    that $\dist_G(u,u') + \dist_G(u',z') = \dist_G(u,u_1) + \dist_T(u_1,z')$.
    Indeed, $u' \in I(u,z')$ because it belongs to a quasi-median of
    $u,v,z$ ; $u_1 \in I(u,z')$ by definition, so the distance between $u$
    and $z'$ passing through $u_1$ and passing through $u'$ are equal:
    $\dist_G(u,u') + \dist_G(u',z) = \dist_G(u,u_1) + \dist_G(u_1,z')$.
    Finally, since $T$ is starshaped (and $z' \in T$), $\dist_G(u_1,z') =
    \dist_T(u_1,z')$.

    Since the quasi-median $u'v'z'$ is an equilateral metric
    triangle, we have
    $$
    \begin{aligned}
    2 \cdot \dist_G(u,v)
    &\ge \dist_G(u,u') + \dist_G(u',z') + \dist_T(z',v) \\
    &= \dist_G(u,u_1) + \dist_G(u_1,z') + \dist_T(z',v) \\
    &= \dist_G(u,u_1) + \dist_T(u_1,v).
    \end{aligned}
    $$
This concludes the proof.
\end{proof}

Stated informally, Lemma \ref{lem_exits} asserts that if $u\in F(x)$ and $T:=
\partial^* F(x)$, then the shortest paths from $u$ to any vertex of $T$ are
``close'' to a path passing via $u_1$ or via $u_2$. The vertices $u_1$ and $u_2$ are called the \emph{exits} of vertex $u$. If $u$ stores the information relative to the exits $u_1$ and $u_2$, then approximate
distances between $u$ and all vertices of $T$ can be easily computed.

\begin{remark}
The inequality of Lemma \ref{lem_exits} is tight as shown by vertices $u$ and $v$ from Figure~\ref{fig_classification_1pp-neighboring}.
\end{remark}

\begin{figure}[htb]
    \begin{minipage}{0.30\linewidth}
        \centering
\includegraphics[width=\linewidth]{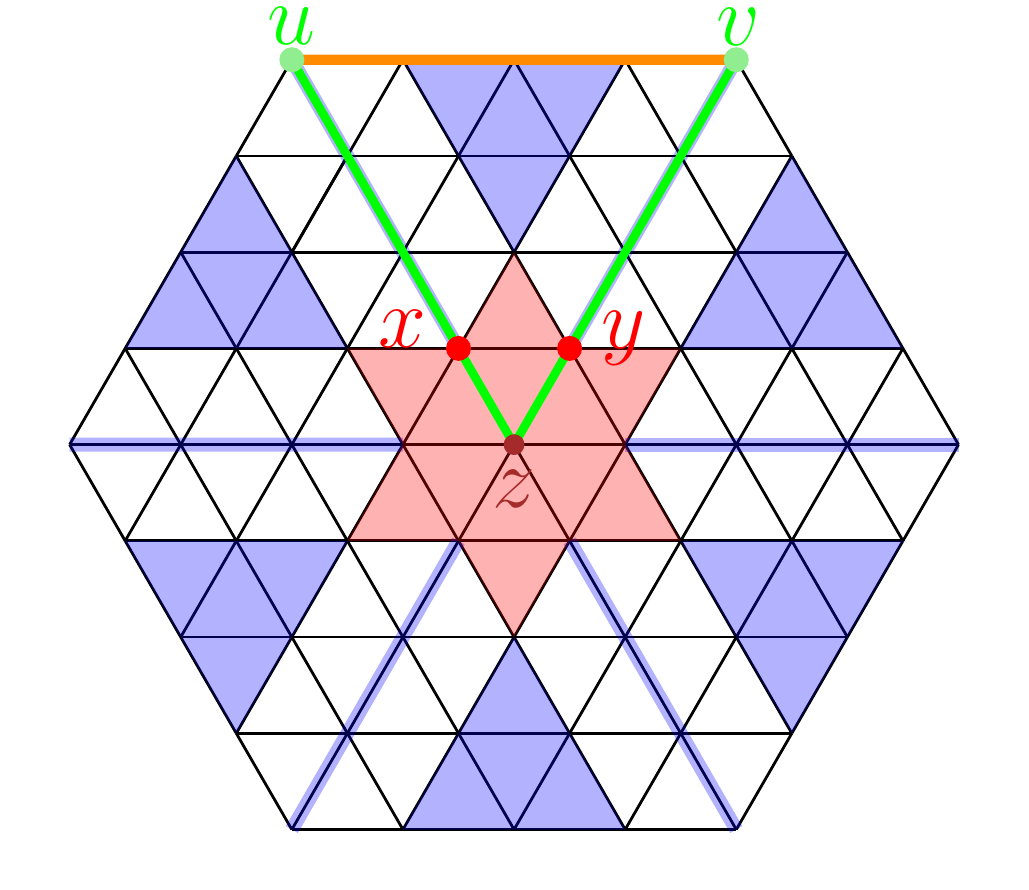}
    \end{minipage}
    \caption{
        \label{fig_classification_1pp-neighboring}
        The vertices $u$ and $v$ coincide with their respective exits.
    }

\end{figure}

\begin{figure}[htb]
    \centering
    \begin{minipage}{0.49\linewidth}
        \centering
        \includegraphics[width=0.53\linewidth]{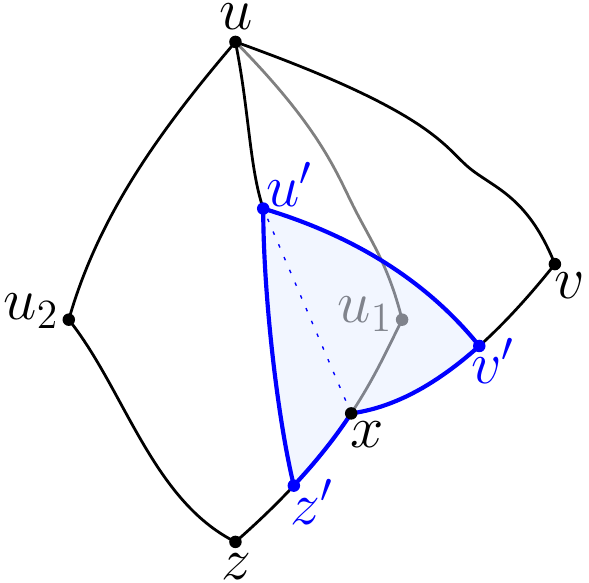}
    \end{minipage}
    \begin{minipage}{0.49\linewidth}
        \centering
        \includegraphics[width=0.6\linewidth]{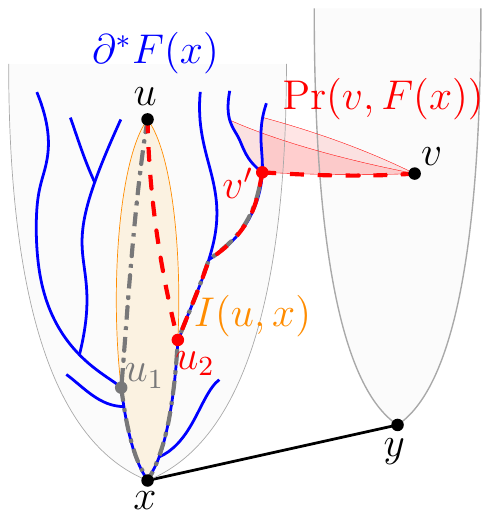}
    \end{minipage}
    \caption{
        \label{fig_2_representants_bordure}
        \label{fig_dist_1pc-neighboring}
        Notations of Lemma \ref{lem_exits} (left) and illustration of the
        entrance and exits used in Lemma \ref{lem_dist_1pc-neighboring}
        (right).
    }
\end{figure}

\section{Shortest paths and classification of pairs of vertices}
\label{sect_classif_vertices_dim2_bridged}

In this section, we characterize the pairs of vertices of $G$ which are connected by a shortest
path passing via the center $z$ of $\St(z)$ (Lemma
\ref{lem_separated_vertices_dim2_bridged}).
We also exhibit the cases for which passing via $z$ can lead to a
multiplicative error $2$ (Lemma
\ref{lem_quasi-separated_vertices_dim2_bridged}).
Finally, we present the cases where our algorithm could make an error of
at most $4$ (Lemma \ref{lem_dist_1pc-neighboring}).
In this last case, our analysis might not be tight.

Let $x$ and $y$ be two vertices of $\St(m)$ and let $(u,v) \in F(x) \times
F(y)$. If $F(x) = F(y)$, then $u$ and $v$ are called \emph{close}.
When $F(x)$ and $F(y)$ are as described in Lemma
\ref{lem_separated_vertices_dim2_bridged}, i.e., if $z \in I(u,v)$, then $u$
and $v$ are called \emph{separated}.
If $F(x)$ and $F(y)$ are 1-neighboring, one of the fibers being a panel and the
other a cone, then $u$ and $v$ are called \emph{1pc-neighboring}.
If $F(x)$ and $F(y)$ denote two 2-neighboring cones, then $u$ and $v$ are
\emph{2cc-neighboring}.
In remaining cases, $u$ and $v$ are said to be \emph{almost separated}.

\subsection{Separated vertices}
For separated vertices $u,v$, clearly $\dist_G(u,z)+\dist_G(z,v)$ is just the distance
$\dist_G(u,v)$. Next lemma establishes which pairs of vertices are separated.
\begin{lemma}
    \label{lem_separated_vertices_dim2_bridged}
    Let $u \in F(x)$ and $v \in F(y)$.
    Then $z \in I(u,v)$ iff $F(x)$ and $F(y)$ are distinct and
    either:
    (i) both are panels and are $k$-neighboring, for $k \ge 2$;
    (ii) one is a panel and the other is a $k$-neighboring cone, for $k \ge 3$;
    (iii) both are cones and are $k$-neighboring, for $k \ge 4$.
\end{lemma}

\begin{proof}
    Consider a quasi-median $u'v'z'$ of the triplet $u,v,z$. The vertex $z$ belongs to a
    shortest $(u,v)$-path if and only if $u' = v' = z' = z$. In that case, let
    $s \in I(u,z)$ and $t \in I(v,z)$ be two neighbors of $z$.
    Since $z$ belongs to a shortest $(u,v)$-path, $s$ and $t$ cannot be
    adjacent. It follows (see Fig. \ref{fig_separated_vertices_dim2_bridged})
    that $F(x)$ and $F(y)$ are $k$-neighboring with:
    (i) $k \ge 2$ if $F(x)$ and $F(y)$ are both panels;
    (ii) $k \ge 3$ if one of $F(x)$ and $F(y)$ is a cone, and the other a panel;
    (iii) $k \ge 4$ if $F(x)$ and $F(y)$ are both cones.

    For the converse implication, we consider the cases where $z$ does not
    belong to a shortest $(u,v)$-path.
    First notice that if $F(x) = F(y)$, then $z$ cannot belong to such a
    shortest path because, then, $\dist_G(u,v) \le \dist_G(u,x) + \dist_G(x,v)
    < \dist_G(u,z) + \dist_G(z,v)$. We now assume that $F(x) \ne F(y)$.
    Three cases have to be considered depending on the type of $F(x)$ and
    $F(y)$.

    \smallskip\noindent
    \textbf{Case 1.} $F(x)$ and $F(y)$ are both panels.
    If $z = z'$, then according to Lemma \ref{lem_MT_structure_K4-free}, $x$
    and $y$ must be the two neighbors of $z$, respectively lying on the
    shortest $(z,u')$- and $(z,v')$-paths, and $x \sim y$, i.e., $F(x)$ and
    $F(y)$ are 1-neighboring. If $z \ne z'$, we consider a vertex $z'' \in I(z,z')$
    adjacent to $z$. Then $z'', x \in I(u,z)$ and, since $u$ belongs to a panel, $z''=x$.
    With the same arguments, we obtain that $z'' = y$. Consequently, $F(x) =
    F(y)$, contrary to our assumption.

    \smallskip\noindent
    \textbf{Case 2.} $F(x)$ is a cone and $F(y)$ is a panel (the symmetric
    case is similar).
    Let $x'$ and $x''$ denote the two neighbors of $x$ in the interval $I(x,z)$.
    If $z = z'$, then by Lemma \ref{lem_MT_structure_K4-free}, $y$ and
    $x'$ (or $x''$) must belong to the deltoid $\Delta(u',v',z')$ and then $x' \sim y$
    (or $x'' \sim y$). It follows that $F(x)$ and $F(y)$ are 2-neighboring.
    If $z \ne z'$, we consider again a neighbor $z''$ of $z$ in $I(z,z')$.
    Since $x', x'' \in I(u,z)$, $z''$ must coincide with $x'$ or with $x''$,
    say $z'' = x'$. Also, $z'' = y$.
    Consequently, $F(x)$ and $F(y)$ are 1-neighboring.

    \smallskip\noindent
    \textbf{Case 3.} $F(x)$ and $F(y)$ are both cones.
    Let $x'$ and $x''$ denote the two neighbors of $x$ in $I(x,z)$, and let
    $y'$ and $y''$ be those of $y$ in $I(z,y)$. Again, if $z = z'$, then $x'$
    (or $x''$) and $y'$ (or $y''$) belong to the deltoid $\Delta(u',v',z')$,
    leading to $x' \sim y'$ and to the fact that
    $F(x)$ and $F(y)$ are 3-neighboring. If $z \ne z'$, we consider
    $z'' \in I(z,z')$, $z'' \sim z$. By arguments
    similar to those used in previous cases, we obtain that $z'' = x' = y'$
    (up to a renaming of the vertices $x''$ and $y''$). It follows that $F(x)$
    and $F(y)$ are 2-neighboring.
\end{proof}

\begin{figure}[htb]
    \centering
    \includegraphics[width=0.65\linewidth]{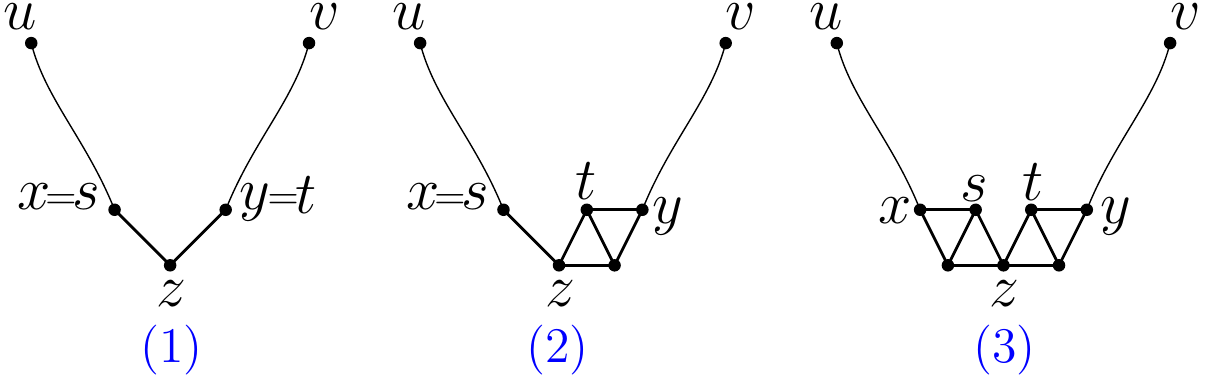}
    \caption{
        \label{fig_separated_vertices_dim2_bridged}
        Illustration of the proof of Lemma \ref{lem_separated_vertices_dim2_bridged}.
    }
\end{figure}

\subsection{Almost separated vertices} The following lemma show that in case of
almost separated vertices, $\dist_G(u,z)+\dist_G(z,v)$ is still a
good approximation of the distance $\dist_G(u,v)$:

\begin{lemma}
    \label{lem_quasi-separated_vertices_dim2_bridged}
    Let $u\in F(x)$ and $v\in F(y)$ be almost separated. Then,
    $
    \dist_G(u,v) \le \dist_G(u,z) + \dist_G(z,v) \le 2 \cdot \dist(u,v).
    $
\end{lemma}

\begin{proof}
    Let $u'v'z'$ be a quasi-median  of the triplet $u,v,z$.
    We have to show that $z = z'$. According to Lemma \ref{lem_separated_vertices_dim2_bridged}, four cases
    must be considered: $F(x)$ and $F(y)$ are two 1-neighboring or
    2-neighboring fibers of distinct types and $F(x)$ and $F(y)$ are two
    3-neighboring cones.

    First, let $F(x)$ and $F(y)$ be 1-neighboring, one of them being
    a panel and the other a cone.
    If $x$ and $y$ belong to a shortest $(u,v)$-path, then $z = z'$.
    Let us assume that this is not the case.
    Then there exists a cone $F(w) \sim F(x), F(y)$ such that $I(u,v) \cap F(w)
    \ne \varnothing$. We claim that, if $z' \notin F(w)$, then $z = z'$.
    Indeed, this directly follows from the fact that $z' \in I(u,z) \cap
    I(v,z)$, $x \notin I(v,z)$ and $y \notin I(u,z)$.
    We now show that $z' \notin F(w)$. Indeed, notice that $x \notin
    I(v,z)$, $y \notin I(u,z)$, and $x, y \in I(w,z)$ imply that $w \notin
    I(u,z) \cup I(v,z)$. By the definition of $u'v'z'$, we obtain that $z' \in
    I(u,z)$. If $z' \in F(w)$, this would contradicts  that  $w \in I(z',z)
    \subseteq I(u,z)$. Hence $z=z'$ in this case as well.

    Let now $F(x)$ and $F(y)$ be  2-neighboring, one of
    them being a panel and the other a cone.
    Suppose that $F(y)$ is the panel and denote by $x_1$ and $x_2$ the two
    neighbors of $x$ in $I(x,z)$. Then in  the same way as before, we show that $z = z'$.
    Indeed, $z' \in I(u,z)$ requires that $z' \in F(x) \cup F(x_1) \cup F(x_2)
    \cup \{z\}$, and $z' \in I(v,z)$ requires that $z' \in F(y) \cup \{z\}$. Consequently,  $z = z'$.

    Finally, let $F(x)$ and $F(y)$ be two 3-neighboring cones. Let $x_1,
    x_2 \in I(x,z)$ and $y_1, y_2 \in I(y,z)$ be distinct from $x$, $y$, $z$,
    and each others. Again, $z' \in I(u,z)$ leads to $z'
    \in F(x) \cup F(x_1) \cup F(x_2) \cup \{z\}$, and $z' \in I(v,z)$ leads to
    $z' \in F(y) \cup F(y_1) \cup F(y_2) \cup \{z\}$. These sets intersect only
    in the vertex $z$, so $z' = z$.
\end{proof}

\subsection{1pc-Neighboring vertices}
Let $F(x)$ be a panel and let $F(y)$ be a cone 1-neighboring $F(x)$.
We set  $T := \partial^* F(x)$. Let $u \in F(x)$ and $v \in F(y)$.
Recall that  by Lemma
\ref{lem_proj_increasing_tree}, $\Pi := \mproj{v}{T}$ induces a tree.
The vertex $v'$ of $\Pi$ closest to $x$ is called the
\emph{entrance of $v$} on the total boundary $T$.
Similarly, two vertices $u_1$ and $u_2$ such as described in Lemma
\ref{lem_exits} are called the \emph{exits of
$u$} on the total boundary $T$. See Fig. \ref{fig_dist_1pc-neighboring}
(right) for an illustration of the notations of this paragraph.

\begin{lemma}
    \label{lem_dist_1pc-neighboring}
    Let $u \in F(x)$ and $v \in F(y)$ be two 1pc-neighboring vertices, where
    $F(x)$ is a panel and $F(y)$ a cone.
    Let $T$, $u_1$, $u_2$ and $v'$ be  as described above. Then,
    $$
    \begin{aligned}
    \dist_G(u,v)
        &\le \min \{
            \dist_G(u,u_1) + \dist_T(u_1,v'),
            \dist_G(u,u_2) + \dist_T(u_2,v')
        \} + \dist_G(v',v)    \\
        &\le 4 \cdot \dist_G(u,v).
    \end{aligned}
    $$
\end{lemma}

\begin{proof}
    Assume $\min \{ \dist_G(u,u_1) + \dist_T(u_1,v'),
    \dist_G(u,u_2) $ $+$ $ \dist_T(u_2,v') \}$ is reached for $u_1$.
    Let $u' \in I(u,v) \cap T$ be closest possible from $\Pi$. Then the
    exact distance between $u$ and $v$ is $\dist_G(u,u') + \dist_G(u',v)$ and
    we have to compare it with $\dist_G(u,u_1) + \dist_T(u_1,v') +
    \dist_G(v',v)$.
    By triangle inequality, we obtain:
    \begin{equation*}
        \dist_G(u,u_1) + \dist_T(u_1,v')
           \le \dist_G(u,u_1) + \dist_T(u_1,u') + \dist_T(u',v').
    \end{equation*}
    Since $v' \in \Pi$, we obtain  $\dist_G(v',v) \le \dist_G(u',v)$.
    By the choice of $u_1$ and by Lemma \ref{lem_exits}, we also know that
    $\dist_G(u,u_1) + \dist_T(u_1,u') \le 2 \cdot \dist_G(u,u')$.
    It remains to compare $\dist_T(u',v')$ with $\dist_G(u',v)$.

    Suppose first that $u'$ and $v'$ belong to a same branch of $T$ and
    consider a quasi-median $u_0'v_0'v_0$ of $u',v',v$.
    Then $I(u',v') \subseteq T$ (because $T$ is starshaped) leads to $u_0',
    v_0' \in T$. The vertex $u'$ being the closest possible to $\Pi$, we have $u' = u_0'$.
    Since $v_0' \in I(v,v')$, $v_0' \in \Pi$.
    Also, since $v'$ is the closest possible to $u'$ vertex, we have $v' = v_0'$.
    Finally, since each quasi-median is an equilateral metric triangle, we conclude that
    $\dist_G(u,u') = \dist_G(u,v')$.
    Consequently, $u' \in \Pi$.
    All  this allows us to obtain the inequalities $\dist_T(u',v') \le
    \dist_G(u',v)$ (Lemma \ref{lem_entrance}), and
    $$
    \begin{array}{rccccclcl}
        \dist_G(u,v)
            &\le& \dist_G(u,u_1)
            &+& \multicolumn{3}{c}{\dist_T(u_1,v')}
            &+& \dist_G(v',v) \\
        &\le& \dist_G(u,u_1)
            &+& \dist_T(u_1,u')
            &+& \dist_T(u',v')
            &+& \dist_G(v',v) \\
        &\le& \multicolumn{3}{c}{2 \cdot \dist_G(u,u')}
            &+& \dist_G(u',v)
            &+& \dist_G(u',v) \\
        &=& \multicolumn{7}{l}{2 \cdot \dist_G(u,v).}
    \end{array}
    $$

    Suppose now that $u'$ and $v'$ belong to distinct branches of $T$, and denote by
    $t$ their nearest common ancestor in this starshaped tree.
    Let $u''$ be the closest vertex to $t$ in the metric projection of $v$ on
    the branch of $u'$.
    Consider a quasi-median $u_0'v_0't_0$ of $u',v',t$.
    Since $T$ is starshaped, $t_0 \in I(u',t) \cap I(v',t)$, and $t$ nearest
    common ancestor of $u'$ and $v'$, we conclude  that $t = t_0$.
    Moreover, $u_0'$ and $v_0'$  belong to the branches of $u'$ and
    of $v'$, respectively.
    We distinguish  five cases, illustrated in Fig.~\ref{figure-1pc}.
    Blue lines correspond to the exact distance $\dist_G(u',v)$, and red ones
    correspond to the approximate distance that we are comparing to it.

\begin{figure}[htb]
    \centering
      \includegraphics[width=0.8\linewidth]{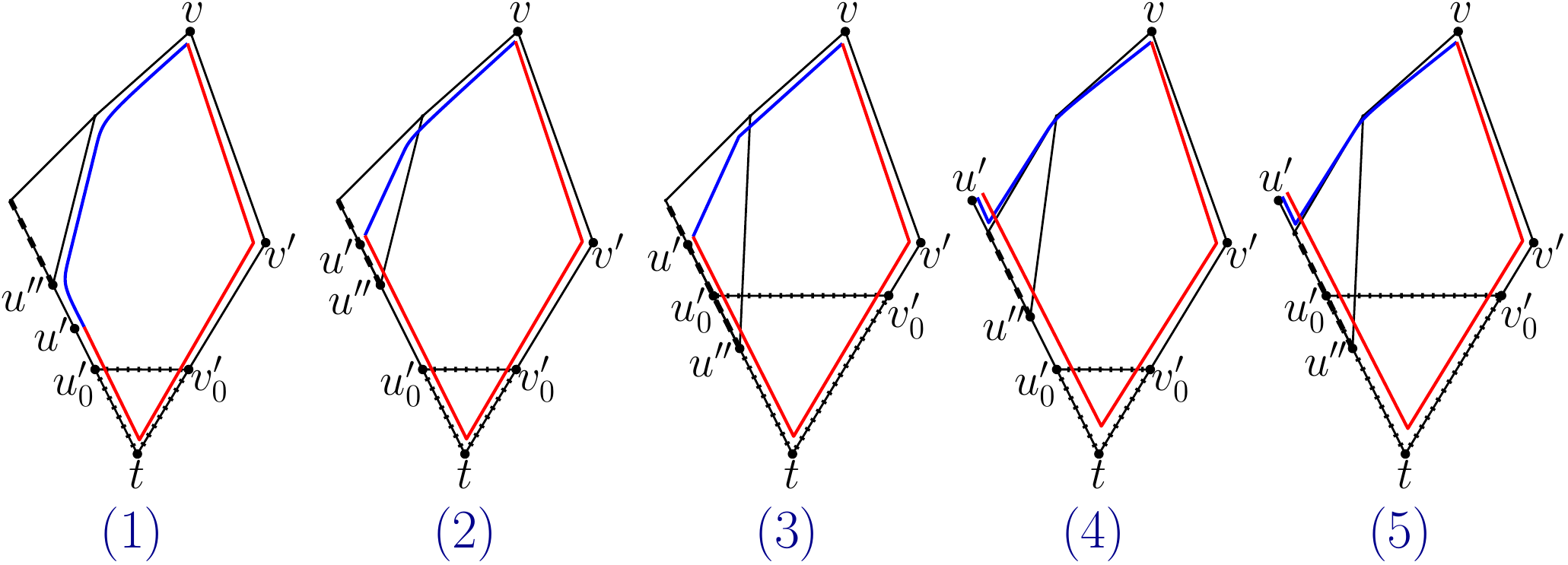}
    \caption{
        \label{figure-1pc}
        Illustration of the proof of Lemma \ref{lem_dist_1pc-neighboring}.
    }
\end{figure}

    In fact, we can notice that the error will be maximal if $u'$ is the
    extremity opposite to $u''$ in the projection of $v$ on the branch of $u'$.
    Indeed, in every considered case the error occurs on the fragment of the path
    between $u'$ and $u''$ that does not belong to the shortest $(u,v)$-path.
    The length of this fragment is maximal in that case.
    We now assert that the two following inequalities hold:

    \begin{claim}
        \label{claim_diamond_bridged}
        $\dist_G(t,v') \le \dist_G(v,u'')$, and $\dist_G(t,u'') \le
        \dist_G(v,v')$.
    \end{claim}

    \begin{proof}%[claim]
        Begin by noticing that $u''$ and $v'$ belong to $I(v,t)$.
        Notice also that $I(t,u'') \cap I(t,v') = \{t\}$, but the intervals
        $I(v,u'')$ and $I(v,v')$ may intersect on some part of $I(v,t)$, not only in $\{v\}$.
        Let $s$ be a furthest from $v$ vertex in this intersection
        ($v$ and $s$ might coincide, as  is the case in the illustration of
        the five cases above, to make it simpler).
        We will prove that $\dist_G(s,u'') = \dist_G(v',t)$, and that
        $\dist_G(s,v') = \dist_G(u'',t)$.

        By Lemma \ref{lem_interval_shape}, the interval $I(s,t)$ induces a burned lozenge, thus a plane graph. We call \emph{bigon}  a subgraph $D$
        of this burned lozenge $I(s,t)$ bounded by two shortest $s,t$-paths, the first one passing via $u''$ and the second one passing
        via $v'$ (for illustration, see Fig. \ref{fig_disque}). Note that each bigon is  a burned lozenge. The area of $D$ is the number
        inner faces (all triangles) of $I(s,t)$ belonging to $D$.

        Assume now that $D$ is a bigon with minimal area.
        The boundary $\partial D$ of $D$ consists of a shortest $(s,u'')$-path $P_1$, shortest $(s,v')$-path $P_2$,
        shortest $(t,u'')$-path $Q_1$, and shortest $(t,v')$-path $Q_2$.
        Notice that each corner of $D$ is either a vertex of $\partial D$ with two neighbors (corner of \emph{type
        1}) or a vertex of $\partial D$ with three  neighbors (corner of 
        \emph{type 2}).
        The two neighbors of a corner of type 1 are adjacent and the three neighbors of a corner of type 2 induce a 3-path.
        The vertices $s$ and $t$ are corners of type 1 because they have exactly
        two neighbors in $I(s,t)$ by Lemma \ref{lem_interval_shape}.
        We now assert that, among the remaining vertices of $\partial D$ of $D$, only $u''$ and
        $v'$ can be corners.
        Indeed, the paths $Q_1$ and $Q_2$ between $t$ and $u'',v'$ are
        convex (because they belong to the branches of a starshaped tree). Consequently, $Q_1$ and $Q_2$ cannot contain corners different from $t,u''$, or
        $v'$. Concerning the paths $P_1$ and $P_2$ between $s$ and $u'',v'$,
        suppose by way of contradiction that one of them, say $P_1$, contains a corner (distinct from $u''$ and $s$).
        If this corner is of type 1, we directly obtain a contradiction with
        the fact that it belongs to a shortest $(s,t)$-path.
        If this corner is of type 2, denote it by $a$, denote by $a'$ its unique  neighbor
        in the interior of $D$, and by $b$ and $c$ its neighbors in $\partial D$. Since $a'$ is adjacent to $b$ and $c$,
        $a'$ belong to a shortest $(s,t)$-path,  obtained by
        replacing $a$ by $a'$. Consequently, we created a bigon $D'$ with two triangles less than
        $D$, contradicting the minimality of $D$.

        Thus, $D$ has at most four corners, $s,t,u'',$ and $v'$. We apply to $D$ the Gauss-Bonnet formula.
        Since $D$ is a burned lozenge,  for every $w \in D \setminus \partial D$, $\kappa(w) = 0$, and for
        every $w \in \partial D \setminus \{s, t, v', u''\}$, $\tau(w) = 0$.
        Since $\tau(s) = \tau(t) = \frac{2\pi}{3}$ by the Gauss-Bonnet formula
        (Theorem~\ref{thm_Gauss-Bonnet}), $\tau(u'') + \tau(v') = \frac{2\pi}{3}$.
        It follows that either $v'$ and $u''$ are both corners of type 2, or one of
        them is a corner of type 1 and the other is not a corner.
        We can easily observe that the second case is impossible because, if
        one vertex is a corner of type 1, it cannot belong to a shortest $(s,t)$-path.
        Thus both $u''$ and $v'$ are corners of type 2. Consequently, $D$ is a full lozenge and we conclude that $\dist_G(s,u'') =
        \dist_G(v',t)$ and $\dist_G(s,v') = \dist_G(u'',t)$. This finishes the proof of the claim.
    \end{proof}

    Returning to the proof of the lemma, from the equalities $\dist_G(v,t) = \dist_G(v,u'') + \dist_G(u'',t) = \dist_G(v,v') +
    \dist_G(v',t) = \dist_G(v,u') + \dist_G(u'',t)$, we deduce that    
    $\dist_G(v,v') + \dist_T(v',u') < 4 \dist_G(v,u')$.
    Consequently,
    $$
    \begin{array}{rcccccl}
        \dist_G(u,v)
            &\le& \dist_G(u,u_1) + \dist_T(u_1,u')
            &+& \dist_T(u',v')
            &+& \dist_G(v',v) \\
        &\le& 2 \cdot \dist_G(u,u')
            &+& 3 \cdot \dist_G(u',v)
            &+& \dist_G(u',v) \\
        &\le& \multicolumn{5}{l}{4 \cdot \dist_G(u,v).}
    \end{array}
    $$
\end{proof}

\begin{figure}[htb]
    \centering
    \includegraphics[width=0.22\linewidth]{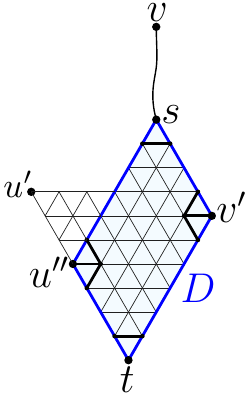}
    \caption{
        \label{fig_disque}
        Illustration of  the proof of Claim \ref{claim_diamond_bridged}.
    }
\end{figure}

\subsection{2cc-Neighboring vertices} Finally, we consider the case of 2cc-neighboring vertices.

\begin{lemma}
    \label{lem_dist_2cc-neighboring}
    Let $u$ and $v$ be two 2cc-neighboring vertices respectively belonging to
    the cones $F(x)$ and $F(y)$. Let $F(w)$ denote the panel 1-neighboring
    $F(x)$ and $F(y)$, and set $T := \partial^* F(w)$.
    Let $u'$ and $v'$ be the respective entrances of $u$ and $v$ on $T$.
    Then
    $$
        \dist_G(u,v)
        \le \dist_G(u,u') + \dist_T(u',v') + \dist_G(v',v)
        \le 4 \cdot \dist_G(u,v).
    $$
\end{lemma}

\begin{proof} 
    First we prove that there exists a shortest $(u,v)$-path traversing the
    panel $F(w)$. Let $w'$ be the second common neighbor of $x$ and $z$ and $w''$
    be the second common neighbor of $y$ and $z$.
    Let $u^*v^*w^*$ be a quasi-median of the triplet $u,v,w$ and let $P(u,v)$
    be a shortest $(u,v)$-path passing via $u^*$ and $v^*$. Finally, let $w_0$
    be a neighbor of $w$ in $I(w,w^*)$. Since $w\in I(u,z)\cap I(v,z)$ and
    $u^*$ belongs to a shortest $(u,w)$-path, we conclude that $w\in I(u^*,z)$.
    Moreover, $w\in I(t,z)$ for any $t\in I(u,u^*)$. Analogously, $w\in I(t,z)$
    for any $t\in I(v,v^*)$. This implies that each of the vertices of
    $I(u,u^*)\cup I(v^*,v)$ belongs either to the panel $F(w)$ or to a cone
    defined by $w$ and a common neighbor of $w$ and $z$.

    While moving along $P(u,v)$ from $u$ to $v$, at some point we have to leave
    the cone $F(x)$. Since two cones cannot be adjacent (Lemma
    \ref{lem_c2c_and_p2p}), necessarily we have to move to a panel. If the
    first change of a fiber happens on the portion of $P(u,v)$ between $u$ and
    $u^*$, then by previous discussion, we conclude that we have to enter the
    panel $F(w)$ and we are done.  So, suppose that $u^*\in F(x)$. Analogously,
    we can suppose that $v^*\in F(y)$. Thus, further we  suppose that $u=u^*$
    and $v=v^*$ (see Figure \ref{fig_2cc-neighboring} for the notations of this
    proof).

    Since $G$ does not contain induced $C_4,C_5$, and $K_4$, one can easily see
    that $w_0\ne x,y$. Since $w_0,x\in I(w,u)$  and $w_0,y\in I(w,v)$, we
    conclude that $w_0\sim x,y$.
    By triangle condition applied to the edge $xw_0$ and $u$, we conclude that
    there exists $x_1\sim x,w_0$ one step closer to $u$ than $x$ and $w_0$.
    Analogously, there exists a vertex $y_1\sim w_0,y$ one step closer to
    $v$. Let $w_1$ be a neighbor of $w_0$ in $I(w_0,w^*)$. If $w_1$ coincides
    with $x_1$ or $y_1$ we will obtain a forbidden $C_4,C_5$, and $K_4$.
    Otherwise, since $x_1,w_1\in I(w_0,u)$ and $y_1,w_1\in I(w_0,v)$ we
    conclude that $w_1\sim x_1,y_1$. Continuing this way, i.e., applying the
    triangle condition and the forbidden graph argument, we will construct the
    vertices $x_i\in I(x_{i-1},u)\cap I(w_{i-1},u),$ $y_i\in I(y_{i-1},v)\cap
    I(w_{i-1},v),$ and $w_i\in I(w_{i-1},w^*)$, such that $x_i\sim
    x_{i-1},w_{i-1}$, $y_i\sim y_{i-1},w_{i-1}$, and $w_i\sim x_i,y_i,$ and
    $w_{i-1}$.
    After $p=\dist_G(w_0,w^*)$ steps, we will have $w_p=w^*$. Let $a$ and $b$
    be the unique neighbors of $w^*$ in $I(w^*,u)$ and $I(w^*,v)$,
    respectively (uniqueness follows from Lemma
    \ref{lem_MT_structure_K4-free}).  Since  $x_p,w_p=w^*$ have the same
    distance to $u$, by triangle condition, $x_p\sim a$. Analogously, we
    conclude that $y_p\sim b$. By Lemma \ref{lem_MT_structure_K4-free}, $a\sim
    b$. The vertices $x_p,a,b,y_p,w_{p-1},w_p$ define a 5-wheel, which cannot
    be induced. But any additional edge leads to a forbidden $K_4$ or $C_4$.
    This contradiction shows that one of the portions of $P(u,v)$ between $u$
    and $u^*$ or between $v^*$ and $v$ contains a vertex $v''$ of the panel
    $F(w)$ adjacent to a vertex of $F(x)$ or of $F(y)$. Clearly, this vertex
    $v''$ must belong to the total boundary $\partial^* F(w)$.

    Then,
    $\dist_G(u,v) = \dist_G(u,v'') + \dist_G(v'',v)$.
    According to Lemma \ref{lem_dist_1pc-neighboring}, the two following
    inequalities hold:
    $$
    \begin{array}{rcccl}
        \dist_G(u,u')   &+& \dist_T(u',v'') &\le& 4 \cdot \dist_G(u,v'') \\
        \dist_T(v'',v') &+& \dist_T(v',v)   &\le& 4 \cdot \dist_G(v'',v)
    \end{array}
    $$
    It follows that
    $$
    \begin{array}{rcccc}
        \dist_G(u,u') + \dist_T(u',v') + \dist_G(v',v)
            &\le& \dist_G(u,u') + \dist_T(u',v'')
            &+&   \dist_T(v'',v') + \dist_G(v',v) \\
            &\le& 4 \cdot \dist_G(u,v'')
            &+&   4 \cdot \dist_G(v'',v) \\
            &=&   \multicolumn{3}{c}{4 \cdot \dist_G(u,v).}
    \end{array}
    $$
\end{proof}

\begin{figure}
    \centering
    \includegraphics[width=0.3\linewidth]{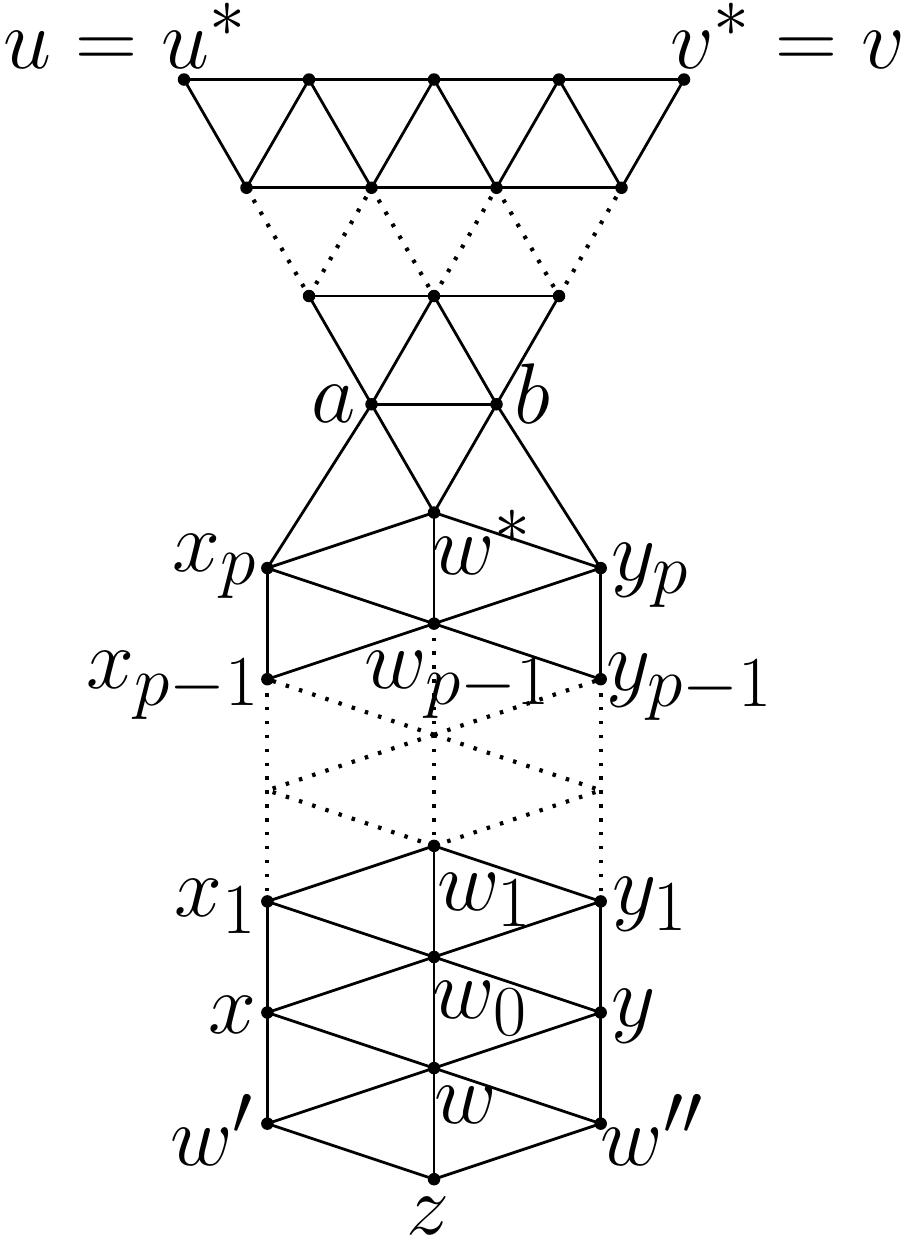}
    \caption{
        \label{fig_2cc-neighboring}
        Illustration of the proof of Lemma \ref{lem_dist_2cc-neighboring}.
    }
\end{figure}

\begin{remark}
    Lemma \ref{lem_dist_2cc-neighboring} covers a case which was not correctly
    considered in the short version of this paper.
\end{remark}

\section{Distance labeling scheme}
\label{sect_dist_labeling_dim2_bridged}

We now describe the $4$-approximate distance labeling scheme for $K_4$-free
bridged graphs. In this section, $G=(V,E)$ is a $K_4$-free
bridged graph with $n$ vertices.

\subsection{Encoding}
We begin  with a brief description of the encoding of  the star $\St(m)$
of a median vertex $m$ of $G$ (in Section
\ref{sect_distdec_bridged}, we explain how to use it to decode the distances).
Then we describe the labels $\L(u)$ given to vertices $u \in V$ by the encoding
algorithm.

\medskip\noindent
\textbf{Encoding of the star.}
Let $m$ be a median vertex of $G$ and let $\St(m)$ be the star of $m$.
The star-label of a vertex $u\in \St(m)$ is denoted by
$\Ls{u}{\St(m)}$. We set $\Ls{m}{\St(m)}:= 0$ (where $0$ will be considered as the empty set
$\varnothing$). Each neighbor of $m$ takes a distinct label in the range $\{1, \ldots, \deg(m)\}$
(interpreted as singletons). The label $\Ls{u}{\St(m)}$ of a vertex $u$ at distance $2$ from $m$ corresponds
to the concatenation of the labels $\Ls{u'}{\St(m)}$ and $\Ls{u''}{\St(m)}$ of
the two neighbors $u'$ and $u''$ of $u$ in $I(u,m)$, i.e., $\Ls{u}{\St(m)}$  is
a set of size 2.

\begin{remark}
    The labels of the vertices of $\St(m)$ not adjacent to $m$ are not necessarily
    unique identifiers of these vertices. Moreover, the labeling of $\St(m)$ does not
    allow to determine adjacency of all pairs of vertices of $\St(m)$.
    Indeed, adjacency queries between vertices encoded by a singleton cannot
    be answered; a singleton label only tells that the corresponding
    vertex is adjacent to $m$, see Fig. \ref{fig_enc_star_bridged}, right.
\end{remark}

\medskip\noindent
\textbf{Encoding of the $K_4$-free bridged graphs.}
Let $u$ denote any vertex of $G$. Let $L_0(u)$ be the unique identifier of $u$. We describe here the part $\L_i(u)$ of the
label of $u$ built at step $i\geq 1$ of the recursion by the encoding procedure
(see \distEncBri).
$\L_i(u)$ consists of three parts: ``\st'', ``\lleft'', and ``\rright''.
The first part $\L_i^{\st}$ contains information relative to the star
$\St(m)$ around the median $m$ chosen in the corresponding step:
the unique identifier $\id(m) =: \L_i^{\st[\med]}(u)$ of $m$ in $G$;
the distance $\dist_G(u,m) =: \L_i^{\st[\distance]}(u)$ between $u$ and $m$;
and a star labeling $\Ls{x}{\St(m)} =: \L_i^{\st[\rootS]}(u)$ of $u$ in
$\St(m)$ (where $x \in \St(m)$ is such that $u \in F(x)$).
This last identifier is used to determine to which type of fibers the vertex  $u$ belongs,
as well as the status (close, separated, 1pc-neighboring, or other) of the pair
$(u,v)$ for any other vertex $v \in V$. Recall that any cone has exactly two 1-neighboring panels.

The two subsequent parts, $\L_i^{\lleft}$ and $\L_i^{\rright}$, depend
whether $F(x)$ is a cone or a panel. If $F(x)$ is a panel, then $\L_i^{\lleft}$ and $\L_i^{\rright}$ contain information relative to the two exits $u_1$ and $u_2$ of $u$
on the total boundary $\partial^* F(x)$ of $F(x)$. The part $\L_i^{\lleft}$ contains (1) an exact distance labeling
$\LDt{u_1}{\partial^* F(x)} =: \L_i^{\lleft[\repD]}(u)$ of $u_1$ in the
total boundary of $F(x)$ and (2) the distance $\dist_G(u,u_1) =:
\L_i^{\lleft[\distance]}(u)$ between $u$ and $u_1$ in $G$. The part $\L_i^{\rright}$
is the same as $\L_i^{\lleft}$ with $u_2$ replacing $u_1$.

If $F(x)$ is a cone, then $\L_i^{\lleft}$ and $\L_i^{\rright}$
contain information relative to the entrances $u_1^+$ and $u_2^+$ of $u$ on (i) the total boundaries $\partial^* F(w_1)$ and $\partial^* F(w_2)$ of the two 1-neighboring fibers $F(w_1)$ and $F(w_2)$ of $F(x)$.
The part $\L_i^{\lleft}$ contains (1) an exact distance labeling
$\LDt{u_1^+}{\partial^* F(w_1)} =: \L_i^{\lleft[\repD]}(u)$ of $u_1^+$ in the
starshaped tree $\partial^* F(w_1)$ (the \dls described in
\cite{FrGaNiWe_trees}, for  example) and (2)  the distance $\dist_G(u,u_1^+)
=: \L_i^{\lleft[\distance]}(u)$ between $u$ and $u_1^+$ in $G$. Finally, the 
part $\L_i^{\rright}$ is the same as $\L_i^{\lleft}$ with $u_2^+$ replacing 
$u_1^+$ and $\partial^* F(w_2)$ instead of $\partial^* F(w_1)$.

\scalebox{0.85}{\begin{algorithm}[H]
    \caption{\label{alg_distencBri}\distEncBri[nolink]($G$, $\L(V)$)}
    \Input{
        A $K_4$-free bridged graph $G = (V,E)$, and a list $\L(V) := \{
        \L(u) := \L_0(u) = (\id(u)) : u \in V \}$ of unique identifiers of its vertices.
    }

    \BlankLine
    \lIf{$V = \{v\}$}{ \Stop }

    \BlankLine
    Find a median vertex $m$ of $G$ \;
    $\Ls{\St(m)}{\St(m)}$ $\leftarrow$ \encStarBri($\St(m)$) \;
    \ForEach{panel $F(x) \in {\mathcal F}_m$}{
        $\LDt{\partial^* F(x)}{\partial^* F(x)}$ $\leftarrow$
        \distEncTree($\partial^* F(x)$)\;

        \ForEach{$u \in F(x)$}{
            Find the exits $u_1$ and $u_2$ of $u$ on $\partial^* F(x)$ \;

            $L^\st \leftarrow
                (\id(m), ~\dist_G(u,m), ~\Ls{x}{\St(m)})$
            \tcp*{
                $L^\st \leftarrow
                    (L^{\st[\med]}, L^{\st[\distance]}, L^{\st[\rootS]})$
            }
            $L^{\lleft} \leftarrow
                (\LDt{u_1}{\partial^* F(x)}, ~\dist_G(u,u_1))$
            \tcp*{
                $L^\lleft \leftarrow
                    (L^{\lleft[\repD]}, L^{\lleft[\distance]})$
            }
            $L^{\rright} \leftarrow
                (\LDt{u_2}{\partial^* F(x)}, ~\dist_G(u,u_2))$
            \tcp*{
                $L^\rright \leftarrow
                (L^{\rright[\repD]}, L^{\rright[\distance]})$
            }
            $\L(u) \leftarrow
            \L(u) \conc (L^\st, ~L^{\lleft}, ~L^{\rright})$ \;
        }
        \distEncBri[]($G[F(x)]$, $\L(F(x))$) \;
    }
    \ForEach{cone $F(x) \in {\mathcal F}_m$}{
        \ForEach{$u \in F(x)$}{
            Find $w_1, w_2 \in \St(m)$ s.t. $F(w_1)$ and $F(w_2)$ are the
            two panels 1-neighboring $F(x)$ \;
            Let $u^+_1$ and $u^+_2$ be the two entrances  of $u$ on $F(w_1)$ and
            $F(w_2)$ \;

            $L^\st \leftarrow
                (\id(m), ~\dist_G(u,m), ~\Ls{x}{\St(m)})$
            \tcp*{
                $L^\st \leftarrow
                (L^{\st[\med]}, L^{\st[\distance]}, L^{\st[\rootS]})$
            }
            $L^{\lleft} \leftarrow
                (\LDt{u^+_1}{\partial^* F(w_1)}, ~\dist_G(u,u^+_1))$
            \tcp*{
                $L^\lleft \leftarrow
                (L^{\lleft[\repD]}, L^{\lleft[\distance]})$
            }
            $L^{\rright} \leftarrow
                (\LDt{u^+_2}{\partial^* F(w_2)}, ~\dist_G(u,u^+_2))$
            \tcp*{
                $L^\rright \leftarrow
                (L^{\rright[\repD]}, L^{\rright[\distance]})$
            }

            $\L(u) \leftarrow
            \L(u) \conc (L^\st, ~L^{\lleft}, ~L^{\rright})$ \;
        }
        \distEncBri[]($G[F(x)]$, $\L(F(x))$) \;
    }
\end{algorithm}}

\subsection{Distance queries}\label{sect_distdec_bridged}
Given the labels $\L(u)$ and $\L(v)$ of two vertices $u$ and $v$,
the distance decoder (see \distDecBri below) starts by determining the state of 
the pair $(u,v)$.
To do so, it looks up for the first median $m$ that separates $u$ and $v$, i.e., such that $u$ and $v$  belong to
distinct fibers with respect to $\St(m)$. More precisely, it looks for the part $i$ of the labels
corresponding to the step in which $m$ became a median.
As noticed in \cite[Section 6.4.6]{ChLaRa_labeling_median}, it is possible to
find this
median vertex $m$ in constant time by adding particular $O(\log^2 n)$ bits
information to the head of each label (consisting of a lowest common ancestor
scheme defined on the tree of median vertices). Once the right parts of label
are found, the decoding function determines that two vertices are
1pc-neighboring
if and only if the identifier (i.e., the star-label in $\St(m)$ of the fiber of
one of the two vertices $u,v$ is strictly included in the identifier of the other).
In that case, the decoding function calls a procedure based on Lemma
\ref{lem_dist_1pc-neighboring} (see \DistPCNeighboring below). 
More precisely, the procedure returns
$$
\min \{
    \dist_G(u,u_1) + \dist_T(u_1,v'),
    \dist_G(u,u_2) + \dist_T(u_2,v')
\} + \dist_G(v',v),
$$ where we assume that $u$ belongs to a panel (and $v$
belongs to a cone), where $u_1$, $u_2$ are contained in the label
parts $\L_i^{\rright}(u)$, $\L_i^{\lleft}(u)$ and $v'$ is contained in the label
part $\L_i^{\rright}(v)$ or $\L_i^{\lleft}(v)$. The distances $\dist_T(u_1,v')$ and
$\dist_T(u_2,v')$ are obtained by decoding the tree distance labels of $u_1$,
$u_2$, and $v'$ in $T$ (also available in these label parts). We also point out
that we assume that $\L_i^{\lleft}(v)$ always contains the
information to get to the panel whose identifier corresponds to the minimum of
the two values identifying the cone of $v$.
The vertices $u$ and $v$ are classified as 2cc-neighboring if and only if
the identifier of their respective fibers intersect in a singleton. In that
case, \distDecBri[nolink] calls the procedure \DistCCNeighboring, based on
Lemma \ref{lem_dist_2cc-neighboring}.
In all the remaining cases (i.e., when $u$ and $v$ are separated or almost
separated), the decoding algorithm will return $\dist_G(u,m) +
\dist_G(v,m)$. By Lemmas \ref{lem_separated_vertices_dim2_bridged} and
\ref{lem_quasi-separated_vertices_dim2_bridged}, this sum is sandwiched between
$\dist_G(u,v)$ and $2\cdot \dist_G(u,v)$.
We now give the two main procedures used by the decoding algorithm
\distDecBri[nolink]: \DistPCNeighboring and \DistCCNeighboring. Note that
\DistPCNeighboring assumes that its first argument $u$ belongs to a panel and
the second $v$ belongs to a cone.

\smallskip
\scalebox{0.85}{\begin{myFunction}
    \Fn{\DistPCNeighboring{$\L_i(u)$, $\L_i(v)$}}{
        $\dir \leftarrow \rright$
        \tcp*[l]{
            If $\L_i^{\st[\rootS]}(u) = \max\{j : j \in
            \L_i^{\st[\rootS]}(v)\}$
        }
        \lIf{$\L_i^{\st[\rootS]}(u) = \min\{j : j \in
            \L_i^{\st[\rootS]}(v)\}$}{
            $\dir \leftarrow \lleft$
        }
        $v' \leftarrow \L_i^{\dir[\repD]}(v)$ \;
        $u_1, ~u_2 \leftarrow
             \L_i^{\lleft[\repD]}(u),
            ~\L_i^{\rright[\repD]}(u)$ \;
        $d_1, ~d_2 \leftarrow
             \distDecTree(v', u_1),
            ~\distDecTree(v', u_2)$ \;
        \Return $\L_i^{\dir[\distance]}(v)
            + \min
            \left\{
                 d_1 + \L_i^{\lleft[\distance]}(u),
                ~d_2 + \L_i^{\rright[\distance]}(u)
            \right\}$ \;	
    }
\end{myFunction}}

Recall that in the following procedure, $u$ and $v$ both belong to cones at
step $i$.

\smallskip
\scalebox{0.85}{\begin{myFunction}
    \Fn{\DistCCNeighboring{$\L_i(u)$, $\L_i(v)$}}{
        $\{x\} \leftarrow \L_i^{\st[\rootS]}(u) \cap \L_i^{\st[\rootS]}(v)$ \;
        $\dirU, ~\dirV \leftarrow \rright, ~\rright$ \;
        \lIf{$x = \min \{ j : j \in \L_i^{\st[\rootS](u)} \}$ }{
            $\dirU \leftarrow \lleft$
        }
        \lIf{$x = \min \{ j : j \in \L_i^{\st[\rootS](v)} \}$ }{
            $\dirV \leftarrow \lleft$
        }
        $u', ~v' \leftarrow
             \L_i^{\dirU[\repD]}(u),
            ~\L_i^{\dirV[\repD]}(v)$ \;
        $d_u, ~d_v \leftarrow
             \L_i^{\dirU[\distance]}(u),
            ~\L_i^{\dirV[\distance]}(v)$ \;
        $d \leftarrow \distDecTree(u', v')$ \;
        \Return $d_u + d + d_v$ \;
    }
\end{myFunction}}

The following algorithm \distDecBri[nolink] finds the first step where the
given vertices $u$ and $v$ have belonged to distinct fibers for the first time.
If they are 1pc-neighboring or 2cc-neighboring at this step, then
\distDecBri[nolink] respectively calls procedure \DistPCNeighboring or
\DistCCNeighboring. Otherwise, it returns the sum of their distances to the
median of the step. The occurring cases are illustrated by 
Figure~\ref{fig_decoding}.

\smallskip
\scalebox{0.85}{\begin{algorithm}[H]
        \caption{\label{alg_distdecBri}\distDecBri[nolink]($\L(u)$,
            $\L(v)$)}
        \Input{The labels $\L(u)$ and $\L(v)$ of two vertices $u$ and $v$ of
            $G$}
        \Output{A value between $\dist_G(u,v)$ and $4 \cdot \dist_G(u,v)$}

        \BlankLine
        \lIf{$\L_0(u) = \L_0(v)$ \tcc*[h]{$u = v$}}{
            \Return $0$
        }

        \BlankLine
        Let $i$ be the greatest integer such that $\L_i^{\st[\med]}(u) =
        \L_i^{\st[\med]}(v)$ \;

        \tcp{If $u$ is in a panel 1-neighboring the cone of $v$}
        \If{$\L_i^{\st[\rootS]}(u) \subsetneq \L_i^{\st[\rootS]}(v)$}{
            \Return $\DistPCNeighboring(\L_i(u), \L_i(v))$ \;
        }
        \tcp{If $v$ is in a panel 1-neighboring the cone of $u$}
        \If{$\L_i^{\st[\rootS]}(v) \subsetneq \L_i^{\st[\rootS]}(u)$}{
            \Return $\DistPCNeighboring(\L_i(v), \L_i(u))$ \;
        }
        \tcp{If $u$ is in a cone 2-neighboring the cone of $v$}
        \If{$|\L_i^{\st[\rootS]}(u) \cap \L_i^{\st[\rootS]}(v)| = 1$
            \and $\L_i^{\st[\rootS]}(u) \ne \L_i^{\st[\rootS]}(v)$
        }{
            \Return $\DistCCNeighboring(\L_i(u), \L_i(v))$ \;
        }
        \tcp{In every other case}
        \Return $\L_i^{\st[\distance]}(u) + \L_i^{\st[\distance]}(v)$ \;
\end{algorithm}}

\begin{figure}[htb]
    \centering
    \includegraphics[width=0.6\linewidth]{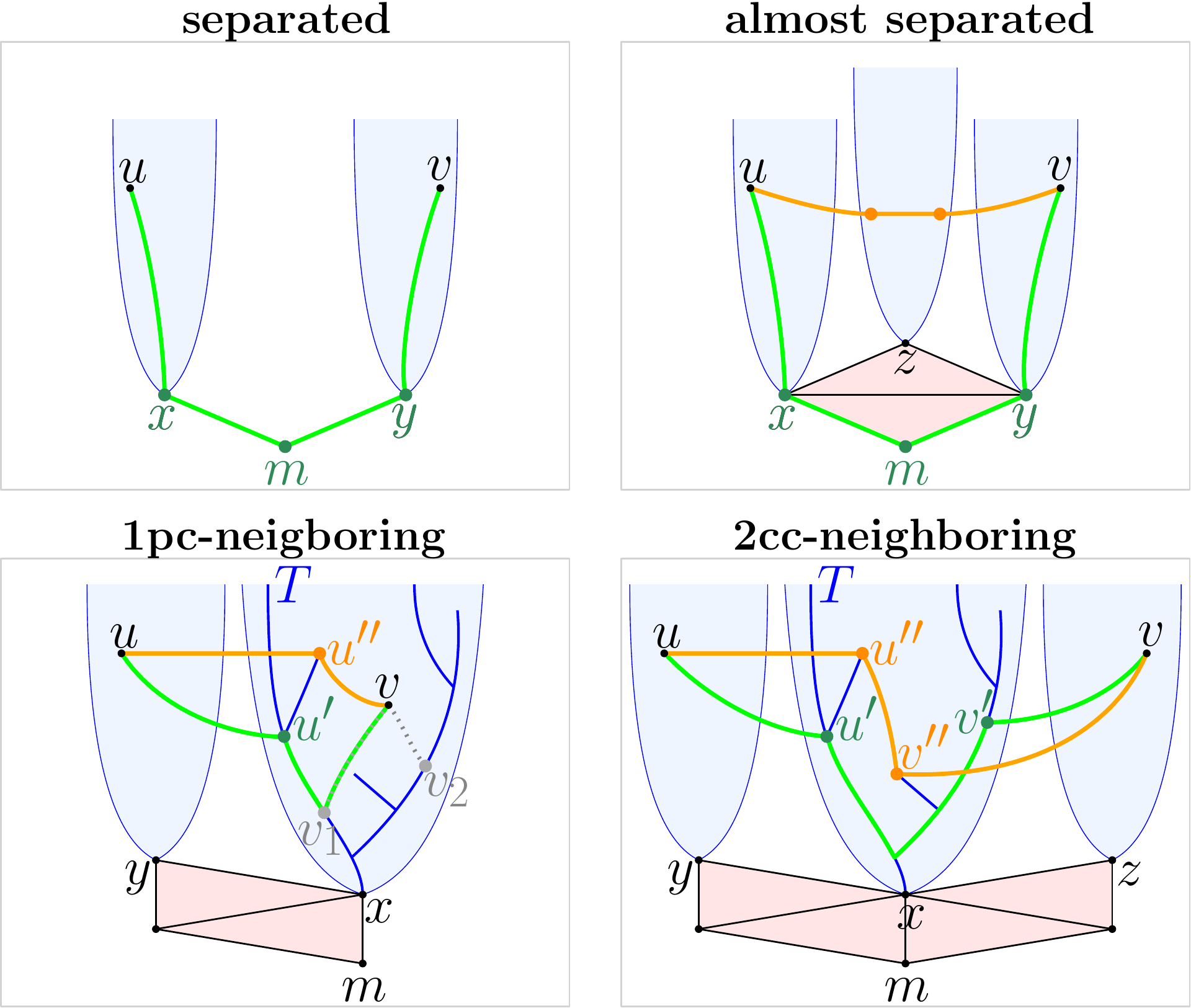}
    \caption{
        \label{fig_decoding}
        Decoding of the distances between separated, almost separated,
        1pc-neighboring, and 2cc-neighboring vertices. An exact shortest path
        is given in orange, and the path considered by our decoding algorithm
        appears in green.
    }
\end{figure}

\subsection{Correctness and complexity}\label{sect_complexite_distance_bridged}
Since by Lemma \ref{lem_small_fibers_dim2_bridged}
the number of vertices in every part is each time divided by 2, the recursion
depth is $O(\log n)$.
At each recursive step, the vertices add to their label a constant number of
information among which the longest consists in a distance labeling scheme for
trees using $O(\log^2 n)$ bits.
It follows that our scheme uses $O(\log^3 n)$ bits for each vertex.
We show, by induction on $n$, that the time complexity is $T(n) := a (n^2 \log
n + n)$ for some constant $a$.
This is trivially true for $n = 1$.
For a graph $G = (V,E)$ on $n$ vertices, we can compute a median vertex $m$ of
$G$, and its partition into fibers $F(x_1), \ldots, F(x_k)$ (with $n_1, \ldots,
n_k$ vertices respectively) in time $O(n^2)$ by using the distance matrix of
$G$.
A vertex $u$ belongs to the total boundary of its fiber $F(x_i)$ if one of its
neighbors belongs to another fiber. This can be determined in time $O(n)$.
Therefore,
the total boundary of all the fibers can be computed in time $O(n^2)$.
Let $u$ be a vertex of a panel $F(x_j)$, the interval $I(x_j, u)$ can be
computed in time $O(n)$, and thus the two exits of $u$ (i.e., the two extremal
vertices of $I(x_j, u) \cap \partial^* F(x_j)$) can also be computed in time
$O(n)$.
Let $v$ be a vertex of a cone $F(x_\ell)$. The projections $\mproj{v}{F}$ and
$\mproj{v}{F'}$ of $v$ on the two panels $1$-neighboring $F(x_\ell)$ can be
computed in time $O(n)$, hence the two entrances of $v$ (i.e., the roots of
$\mproj{v}{F} \cap \partial^* F$ and $\mproj{v}{F'} \cap \partial^* F'$) can
also be computed in time $O(n)$.
Consequently, the first step of the induction requires time $a n^2$.
By induction, computing the label of all $u$ in $F(x_i)$ takes time $T(n_i)$.
Therefore, $T(n) = an^2 + \sum_{i=1}^k T(n_i)$. By Lemma
\ref{lem_small_fibers_dim2_bridged}, $n_i \le n / 2$ for all $i \in
\{1,\ldots,k\}$. Consequently,
$$
\begin{array}{rcl}
  T(n) &=   &a n^2 + \sum_{i=1}^k (a n_i^2 \log n_i + n_i) \\
       &\le &a n^2 + a n^2 \log(\frac{n}{2}) + a n         \\
       &\le &a (n^2 \log n + n).
\end{array}
$$
We conclude that the total complexity of Algorithm \ref{alg_distencBri} is
$O(n^2 \log n)$.

That the decoding algorithm returns distances with a multiplicative error
at most $4$ directly follows from Lemmas
\ref{lem_separated_vertices_dim2_bridged} and
\ref{lem_quasi-separated_vertices_dim2_bridged} for separated and almost
separated vertices, and from Lemmas \ref{lem_dist_1pc-neighboring} and
\ref{lem_dist_2cc-neighboring} for the 1pc-neighboring and 2cc-neighboring
vertices. Those results are based on Lemmas \ref{lem_entrance} and
\ref{lem_exits} that respectively indicate the entrances and exits to store in
total boundaries of panels. This concludes the proof of Theorem
\ref{thm_dist_labeling_bridged}.

\section{Conclusion}

We would like to finish this paper with some open questions.
First of all, the problem of finding a polylogarithmic (approximate) distance
labeling scheme for general bridged graphs remains open. We formulate it in the
following way:

\begin{question}
    Do there exist constants $c$ and $b$ such that any bridged graph $G$
    admits a $c$-approximate distance labeling scheme with labels of size
    $O(\log^b n)$ ?
\end{question}

One of the first obstacles in adapting our labeling scheme to general bridged graphs is
that it is not clear how to define the star and the fibers. In all bridged graphs,
since the neighborhood $N[z]$ of a vertex $z$ is convex, the metric projection of
any vertex $u$ on $N[z]$ induces a clique. Therefore, we could define the fiber $F(C)$
for a  clique $C$ of $N[z]$ as the set of all $u$ having $C$ as metric
projection on $N[z]$. This induces a partitioning of $V(G) \setminus N[z]$,
however the
interaction between different fibers of $N[z]$ seems intricate.

The same question can be asked for bridged graphs of constant clique-size and
for hyperbolic bridged graphs (via a result of \cite{ChChHiOs}, those are the
bridged graphs in which all deltoids have
constant size). A positive result would be interesting since Gavoille and Ly
\cite{GaLy_hyperbolic} established that general graphs of bounded hyperbolicity
do not admit poly-logarithmic distance labeling schemes unless we allow a
multiplicative error of order $\Omega(\log \log n)$, at least.

\begin{question}
    Do there exist linear functions $f$ and $g$ such that every
    $\delta$-hyperbolic bridged graph $G$ admits a $f(\delta)$-approximate distance
    labeling scheme with labels of size $O(\log^{g(\delta)} n)$ ?
\end{question}

In the full version of \cite{ChLaRa_labeling_median}, we managed to
encode the cube-free median graphs in $O(n \log n)$ time instead of $O(n^2 \log
n)$. This improvement uses a recent result of \cite{BeChChVa_median}
allowing to compute a median vertex of a median graph in linear time. We also
compute in cube-free median graphs the partition into fibers, the gates
(equivalent of the entrances) and the  imprints (equivalent to the exits) in
fibers in linear time, with a BFS-like algorithm. Altogether,  this improvement
of the preprocessing time for cube-free median graphs was technically
non-trivial.
In the case of $K_4$-free bridged graph, a similar result can be expected.
The first step will be to design a linear-time algorithm for computing
medians in $K_4$-free bridged graphs. For planar $K_4$-free bridged graphs, such an algorithm was
described in \cite{ChFaVa}.

\subsubsection*{Acknowledgment} We would like to acknowledge the referees for their very careful reading of the manuscript, numerous insightful comments, and useful suggestions. This work  was supported by ANR project DISTANCIA
(ANR-17-CE40-0015).

\bibliographystyle{plainurl}
\bibliography{biblio-bridged}

\begin{thebibliography}{10}

\bibitem{AnFa}
R.P Anstee and M~Farber.
\newblock On bridged graphs and cop-win graphs.
\newblock {\em Journal of Combinatorial Theory, Series B}, 44(1):22--28, 1988.
\newblock URL:
  \url{https://www.sciencedirect.com/science/article/pii/0095895688900937},
  \href {http://dx.doi.org/https://doi.org/10.1016/0095-8956(88)90093-7}
  {\path{doi:https://doi.org/10.1016/0095-8956(88)90093-7}}.

\bibitem{BaCh_Helly}
H.-J. Bandelt and V.~Chepoi.
\newblock A {H}elly theorem in weakly modular space.
\newblock {\em Discrete Math.}, 160(1-3):25--39, 1996.

\bibitem{BaCh_wma}
H.-J. Bandelt and V.~Chepoi.
\newblock The algebra of metric betweenness ii: axiomatics of weakly median
  graphs.
\newblock {\em Europ. J. Combin.}, 29:676--700, 2008.

\bibitem{BaCh_survey}
H.-J. Bandelt and V.~Chepoi.
\newblock Metric graph theory and geometry: a survey.
\newblock {\em Contemporary Mathematics}, 453:49--86, 2008.

\bibitem{BaCh_dwmg}
Hans{-}J{\"{u}}rgen Bandelt and Victor Chepoi.
\newblock Decomposition and\emph{l}\({}_{\mbox{1}}\)-embedding of weakly median
  graphs.
\newblock {\em Eur. J. Comb.}, 21(6):701--714, 2000.
\newblock URL: \url{https://doi.org/10.1006/eujc.1999.0377}, \href
  {http://dx.doi.org/10.1006/eujc.1999.0377}
  {\path{doi:10.1006/eujc.1999.0377}}.

\bibitem{BaGa_data_structures}
F.~Bazzaro and C.~Gavoille.
\newblock Localized and compact data-structure for comparability graphs.
\newblock {\em Discr. Math.}, 309:3465--3484, 2009.

\bibitem{BeChChVa_median}
L.~B{\'e}n{\'e}teau, J.~Chalopin, V.~Chepoi, and Y.~Vax{\`e}s.
\newblock Medians in median graphs and their cube complexes in linear time.
\newblock {\em J. Comput. Syst. Sci.}, 126:80--105, 2022.

\bibitem{ChChHiOs}
J.~Chalopin, V.~Chepoi, H.~Hirai, and D.~Osajda.
\newblock Weakly modular graphs and nonpositive curvature.
\newblock {\em Memoirs of AMS}, 268(1309):159 pp., 2020.

\bibitem{Chepoi_metric_triangles}
V.~Chepoi.
\newblock Classification of graphs by means of metric triangles.
\newblock {\em Metody Diskret. Analiz.}, pages 75--93, 96, 1989.

\bibitem{Ch_CAT}
V.~Chepoi.
\newblock Graphs of some {CAT(0)} complexes.
\newblock {\em Adv. Appl. Math.}, 24:125--179, 2000.

\bibitem{ChDrVa_labeling}
V.~Chepoi, F.~F. Dragan, and Y.~Vax{\`{e}}s.
\newblock Distance and routing labeling schemes for non-positively curved plane
  graphs.
\newblock {\em J. Algorithms}, 61:60--88, 2006.

\bibitem{Ch_bridged}
Victor Chepoi.
\newblock Bridged graphs are cop-win graphs: an algorithmic proof.
\newblock {\em J. Combin. Theory Ser. B}, 69(1):97--100, 1997.

\bibitem{ChDrVa_SODA}
Victor Chepoi, Feodor~F. Dragan, and Yann Vax{\`{e}}s.
\newblock Center and diameter problems in plane triangulations and
  quadrangulations.
\newblock In David Eppstein, editor, {\em Proceedings of the Thirteenth Annual
  {ACM-SIAM} Symposium on Discrete Algorithms, January 6-8, 2002, San
  Francisco, CA, {USA}}, pages 346--355. {ACM/SIAM}, 2002.
\newblock URL: \url{http://dl.acm.org/citation.cfm?id=545381.545427}.

\bibitem{ChDrVa2005}
Victor Chepoi, Feodor~F. Dragan, and Yann Vax{\`{e}}s.
\newblock Distance-based location update and routing in irregular cellular
  networks.
\newblock In Lawrence Chung and Yeong{-}Tae Song, editors, {\em Proceedings of
  the 6th {ACIS} International Conference on Software Engineering, Artificial
  Intelligence, Networking and Parallel/Distributed Computing {(SNPD} 2005),
  May 23-25, 2005, Towson, Maryland, {USA}}, pages 380--387. {IEEE} Computer
  Society, 2005.
\newblock URL: \url{https://doi.org/10.1109/SNPD-SAWN.2005.32}, \href
  {http://dx.doi.org/10.1109/SNPD-SAWN.2005.32}
  {\path{doi:10.1109/SNPD-SAWN.2005.32}}.

\bibitem{ChDrVa_wn}
Victor Chepoi, Feodor~F. Dragan, and Yann Vax{\`{e}}s.
\newblock Addressing, distances and routing in triangular systems with
  applications in cellular networks.
\newblock {\em Wirel. Networks}, 12(6):671--679, 2006.
\newblock URL: \url{https://doi.org/10.1007/s11276-006-6527-0}, \href
  {http://dx.doi.org/10.1007/s11276-006-6527-0}
  {\path{doi:10.1007/s11276-006-6527-0}}.

\bibitem{ChFaVa}
Victor Chepoi, Cl{\'{e}}mentine Fanciullini, and Yann Vax{\`{e}}s.
\newblock Median problem in some plane triangulations and quadrangulations.
\newblock {\em Comput. Geom.}, 27(3):193--210, 2004.
\newblock URL: \url{https://doi.org/10.1016/j.comgeo.2003.11.002}, \href
  {http://dx.doi.org/10.1016/j.comgeo.2003.11.002}
  {\path{doi:10.1016/j.comgeo.2003.11.002}}.

\bibitem{ChLaRa_labeling_median}
Victor Chepoi, Arnaud Labourel, and S{\'{e}}bastien Ratel.
\newblock Distance and routing labeling schemes for cube-free median graphs.
\newblock {\em Algorithmica}, 83(1):252--296, 2021.
\newblock URL: \url{https://doi.org/10.1007/s00453-020-00756-w}, \href
  {http://dx.doi.org/10.1007/s00453-020-00756-w}
  {\path{doi:10.1007/s00453-020-00756-w}}.

\bibitem{ChepoiOsajda}
Victor Chepoi and Damian Osajda.
\newblock Dismantlability of weakly systolic complexes and applications.
\newblock {\em Trans. Amer. Math. Soc.}, 367(2):1247--1272, 2015.
\newblock URL: \url{http://dx.doi.org/10.1090/S0002-9947-2014-06137-0}, \href
  {http://dx.doi.org/10.1090/S0002-9947-2014-06137-0}
  {\path{doi:10.1090/S0002-9947-2014-06137-0}}.

\bibitem{CoVa_bounded_cw}
B.~Courcelle and R.~Vanicat.
\newblock Query efficient implementation of graphs of bounded clique-width.
\newblock {\em Discrete Appl. Math.}, 131:129--150, 2003.

\bibitem{Elsner2009-isometries}
Tomasz Elsner.
\newblock Isometries of systolic spaces.
\newblock {\em Fund. Math.}, 204(1):39--55, 2009.
\newblock URL: \url{http://dx.doi.org/10.4064/fm204-1-3}, \href
  {http://dx.doi.org/10.4064/fm204-1-3} {\path{doi:10.4064/fm204-1-3}}.

\bibitem{FaJa_convexity}
M.~Farber and R.~E. Jamison.
\newblock On local convexity in graphs.
\newblock {\em Discr. Math.}, 66(3):231--247, 1987.

\bibitem{FrGaNiWe_trees}
O.~Freedman, P.~Gawrychowski, P.~K. Nicholson, and Oren Weimann.
\newblock Optimal distance labeling schemes for trees.
\newblock In {\em PODC}, pages 185--194. ACM, 2017.

\bibitem{GaKaKaPaPe_approx}
C.~Gavoille, M.~Katz, N.~A. Katz, C.~Paul, and D.~Peleg.
\newblock Approximate distance labeling schemes.
\newblock In {\em ESA}, pages 476--487. Springer, 2001.

\bibitem{GaLy_hyperbolic}
C.~Gavoille and O.~Ly.
\newblock Distance labeling in hyperbolic graphs.
\newblock In {\em International Symposium on Algorithms and Computation}, pages
  1071--1079. Springer, 2005.

\bibitem{GaPa_decomposition}
C.~Gavoille and C.~Paul.
\newblock Distance labeling scheme and split decomposition.
\newblock {\em Discr. Math.}, 273:115--130, 2003.

\bibitem{GaPa_interval_graphs}
C.~Gavoille and C.~Paul.
\newblock Optimal distance labeling for interval graphs and related graph
  families.
\newblock {\em SIAM J. Discr. Math.}, 22:1239--1258, 2008.

\bibitem{GaPePeRa_distance}
C.~Gavoille, D.~Peleg, S.~P{\'{e}}renn{\`{e}}s, and R.~Raz.
\newblock Distance labeling in graphs.
\newblock {\em J. Algorithms}, 53:85--112, 2004.

\bibitem{GaUz_distance}
P.~Gawrychowski and P.~Uznanski.
\newblock A note on distance labeling in planar graphs.
\newblock {\em CoRR}, abs/1611.06529, 2016.
\newblock \href {http://arxiv.org/abs/1611.06529} {\path{arXiv:1611.06529}}.

\bibitem{GeSh}
S.M. Gersten and H.B. Short.
\newblock Small cancellation theory and automatic groups.
\newblock {\em Invent. Math.}, 102:305--334, 1990.

\bibitem{Gr}
Misha Gromov.
\newblock Hyperbolic groups.
\newblock In {\em Essays in group theory}, volume~8 of {\em Math. Sci. Res.
  Inst. Publ.}, pages 75--263. Springer, New York, 1987.
\newblock URL: \url{https://doi.org/10.1007/978-1-4613-9586-7_3}, \href
  {http://dx.doi.org/10.1007/978-1-4613-9586-7_3}
  {\path{doi:10.1007/978-1-4613-9586-7_3}}.

\bibitem{Haglund}
Fr\'ed\'eric Haglund.
\newblock Complexes simpliciaux hyperboliques de grande dimension.
\newblock {\em Prepublication Orsay}, (71), 2003.

\bibitem{HoOs}
N.~Hoda and D.~Osajda.
\newblock Two-dimensional systolic complexes satisfy property a.
\newblock {\em Internat. J. Algebra Comput.}, 28:1247--1254, 2018.

\bibitem{HuOs_ms}
Jingyin Huang and Damian Osajda.
\newblock Metric systolicity and two-dimensional {A}rtin groups.
\newblock {\em Math. Ann.}, 374(3-4):1311--1352, 2019.

\bibitem{HuOs_Artin}
Jingyin Huang and Damian Osajda.
\newblock Large-type {A}rtin groups are systolic.
\newblock {\em Proc. Lond. Math. Soc.}, 120(1):95--123, 2020.

\bibitem{JaSw}
T.~Januszkiewicz and J.~{\'{S}}wi{\c{a}}tkowski.
\newblock Simplicial nonpositive curvature.
\newblock {\em Publications Math{\'e}matiques de l'Institut des Hautes
  {\'E}tudes Scientifiques}, 104(1):1--85, 2006.

\bibitem{JanuszkiewiczSwiatkowski2007}
Tadeusz Januszkiewicz and Jacek \'Swi{\c a}tkowski.
\newblock Filling invariants of systolic complexes and groups.
\newblock {\em Geom. Topol.}, 11:727--758, 2007.
\newblock URL: \url{http://dx.doi.org/10.2140/gt.2007.11.727}, \href
  {http://dx.doi.org/10.2140/gt.2007.11.727}
  {\path{doi:10.2140/gt.2007.11.727}}.

\bibitem{LySc_group_theory}
R.~C. Lyndon and P.~E Schupp.
\newblock {\em Combinatorial Group Theory}.
\newblock Springer, 2015.

\bibitem{OsaPrzy2009}
Damian Osajda and Piotr Przytycki.
\newblock Boundaries of systolic groups.
\newblock {\em Geom. Topol.}, 13(5):2807--2880, 2009.
\newblock URL: \url{https://doi.org/10.2140/gt.2009.13.2807}, \href
  {http://dx.doi.org/10.2140/gt.2009.13.2807}
  {\path{doi:10.2140/gt.2009.13.2807}}.

\bibitem{Peleg00}
D.~Peleg.
\newblock {\em {D}istributed {C}omputing: {A} {L}ocality-{S}ensitive
  {A}pproach}.
\newblock SIAM, 2000.

\bibitem{Po_bridged1}
Norbert Polat.
\newblock On infinite bridged graphs and strongly dismantlable graphs.
\newblock {\em Discrete Math.}, 211(1-3):153–--166, 2000.

\bibitem{Po_bridged2}
Norbert Polat.
\newblock On isometric subgraphs of infinite bridged graphs and geodesic
  convexity.
\newblock {\em Discrete Math.}, 244(1-3):399–--416, 2002.

\bibitem{Prytula2017}
Tomasz Prytu{\l}a.
\newblock Hyperbolic isometries and boundaries of systolic complexes.
\newblock {\em arXiv preprint arXiv:1705.01062}, 2017.

\bibitem{Pr3}
Piotr Przytycki.
\newblock The fixed point theorem for simplicial nonpositive curvature.
\newblock {\em Mathematical Proceedings of the Cambridge Philosophical
  Society}, 144:683 -- 695, 2008.
\newblock \href {http://dx.doi.org/10.1017/S0305004107000989}
  {\path{doi:10.1017/S0305004107000989}}.

\bibitem{Ratel}
S.~Ratel.
\newblock {\em Densit\'e, {V}{C}-dimension et \'etiquetages de graphes}.
\newblock Aix-Marseille Universit\'e, 2019.

\bibitem{SoCh_convexification}
V.~Soltan and V.~Chepoi.
\newblock Conditions for invariance of set diameters under $d$-convexification
  in a graph.
\newblock {\em Cybernetics}, 19(6):750--756, 1983.

\bibitem{Thorup_approx_distance}
M.~Thorup.
\newblock Compact oracles for reachability and approximate distances in planar
  digraphs.
\newblock {\em J. of ACM}, 51:993--1024, 2004.

\bibitem{Weetman}
G.~M. Weetman.
\newblock A construction of locally homogeneous graphs.
\newblock {\em J. London Math. Soc. (2)}, 50:68--86, 1994.

\bibitem{Wise}
D.~T. Wise.
\newblock Sixtolic complexes and their fundamental groups.
\newblock 2003.

\end{thebibliography}

\newpage
\section*{Appendices}
\addcontentsline{toc}{section}{Appendices}
\setcounter{tocdepth}{0}
\appendix

\section{Glossary}
\label{ref_glo}
In the following glossary, $G$ denotes a graph of vertex set $V$ and edge set
$E$ ; unless it is explicitly defined otherwise, $T$ denotes a starshaped tree
rooted at $z$; $H = (V(H), E(H))$ denotes a subgraph of $G$.

\medskip

{
    \renewcommand{\arraystretch}{1.2}
    \begin{tabular}{p{0.33\linewidth}p{0.67\linewidth}}
        \hline
        \textbf{Notions and notations}
            & \textbf{Definitions}
            \\
        \hline
        Ball $B_k(S)$
            & $\{ x \in V : \exists s \in S, \dist_G(x,s) \le k \}$.
            \\
        Boundary $\partial_{y} F(x)$
            & $\{ u \in F(x) : \exists v \in F(y), uv \in E \}$.
            \\
        Closed neighborhood $N[u]$ of $u$
            & $\{ v \in V : uv \in E \} \cup \{u\}$.
            \\
        Cone $F(x)$ w.r.t. to $\St(z)$
            & Fiber $F(x)$ w.r.t. $\St(z)$ with $\dist_G(x,z) = 2$.
            \\
        Convex subgraph $H$
            & $\forall u, v \in V(H)$, $I(u,v) \subseteq V(H)$.
            \\
        Distance $\dist_G(u,v)$
            & Number of edges on a shortest $(u,v)$-path of $G$.
            \\
        Entrance of $u$ on $T := \partial^* F(x)$
            & Closest vertex to $x$ in $\mproj{u}{T}$.
            \\
        Exit of $u$ on $T := \partial^* F(x)$
            & Extremal vertex w.r.t. $z$ in an increasing path of $I(u,z) \cap
            T$.
            \\
        Extremal vertex of $I(u,z) \cap T$.
            & First convex corner in an increasing path starting at $z$.
            \\
        Fiber $F(x)$ w.r.t. $H$
            & $\{u \in V : x \text{ is the gate of $u$ in } H \}$.
            \\
        Increasing path $P$ of $T$
            & Path entirely contained in a single branch of $T$.
            \\
        Interval $I(u,v)$
            & $\{ w \in V : \dist_G(u,v) = \dist_G(u,w) + \dist_G(w,v) \}$.
            \\
        Isometric subgraph $H$
            & $\forall u, v \in V(H), \dist_H(u,v) = \dist_G(u,v)$.
            \\
        Locally convex subgraph $H$
            & $\forall u, v \in V(H)$, with $\dist_G(u,v) \le 2$, $I(u,v)
            \subseteq V(H)$.
            \\
        Median vertex $m$
            & Vertex minimizing $u \mapsto \sum_{v \in V} \dist_G(u,v)$.
            \\
        Metric projection $\mproj{x}{H}$
            & $\{ u \in V(H) : \forall v \in V(H), \dist_G(v,x) \ge
            \dist_G(u,x)
            \}$.
            \\
        Metric triangle $u_1u_2u_3$ of $G$
            & $u_1, u_2, u_3 \in V$ s.t. $\forall i, j, k \in
            \{1,2,3\}$, $I(u_i,u_j) \cap I(u_j, u_k) = \{u_j\}$.
            \\
        Panel $F(x)$ w.r.t. to $\St(z)$
            & Fiber $F(x)$ w.r.t. $\St(z)$ with $\dist_G(x,z) = 1$.
            \\
        Star $\St(z)$
            & Union of $N[z]$ and of all the triangles derived from $N[z]$
            using (TC).
            \\
        Starshaped $H$ w.r.t. $z$
            & $\forall u \in V(H), I(u,z) \subseteq V(H)$.
            \\
        Starshaped tree $T \subseteq G$ w.r.t. $z$
            & $\forall u \in V(T)$, $I(u,z)$ is a path of $T$.
            \\
        Total boundary $\partial^* F(x)$
            & $\bigcup_{y \sim x} \partial_{y} F(x)$.
            \\
        Triangle condition (TC)
            & $\forall u,v,w \in V$ with $k := \dist_G(u,v) = \dist_G(u,w)$,
            and $vw \in E$, $\exists x \in V$ s.t. $\dist_G(u,x) = k - 1$
            and $xv, xw \in E$.
            \\
        \hline
    \end{tabular}
}

\end{document}